\documentclass{amsart}

\usepackage[latin1]{inputenc}
\usepackage{amsfonts}
\usepackage{amsmath}
\usepackage{amsthm}
\usepackage{amssymb}
\usepackage{latexsym}
\usepackage{enumerate}

\newtheorem{theorem}{Theorem}

\newtheorem{corollary}[theorem]{Corollary}

\newtheorem{definition}[theorem]{Definition}

\newtheorem{lemma}[theorem]{Lemma}

\newtheorem{proposition}[theorem]{Proposition}
\newtheorem{remark}[theorem]{Remark}

\numberwithin{equation}{section}

\begin{document}

\newcommand{\cc}{\mathfrak{c}}
\newcommand{\N}{\mathbb{N}}
\newcommand{\C}{\mathbb{C}}
\newcommand{\Q}{\mathbb{Q}}
\newcommand{\R}{\mathbb{R}}
\newcommand{\T}{\mathbb{T}}
\newcommand{\st}{*}
\newcommand{\PP}{\mathbb{P}}
\newcommand{\lin}{\left\langle}
\newcommand{\rin}{\right\rangle}
\newcommand{\SSS}{\mathbb{S}}
\newcommand{\forces}{\Vdash}
\newcommand{\dom}{\text{dom}}
\newcommand{\osc}{\text{osc}}
\newcommand{\F}{\mathcal{F}}
\newcommand{\A}{\mathcal{A}}
\newcommand{\B}{\mathcal{B}}
\newcommand{\I}{\mathcal{I}}
\newcommand{\CC}{\mathcal{C}}
\newcommand{\GG}{\mathbb{G}}
\newcommand{\UU}{\mathbb{U}}
\newcommand{\BB}{\mathbb{B}}

\title[On $\R$-embeddability of almost disjoint families]{On $\R$-embeddability of almost disjoint families and Akemann-Doner C*-algebras}

\author{Osvaldo Guzm\'{a}n}
\address{Centro de Ciencas Matem\'aticas\\
UNAM\\
A.P. 61-3, Xangari, Morelia, Michoac\'an\\
58089, M\'exico}
\email{oguzman@matmor.unam.mx}

\author{Michael Hru\v{s}\'{a}k}
\address{Centro de Ciencas Matem\'aticas\\
UNAM\\
A.P. 61-3, Xangari, Morelia, Michoac\'an\\
58089, M\'exico}
\email{michael@matmor.unam.mx}

\author{Piotr Koszmider}
\address{Institute of Mathematics, Polish Academy of Sciences,
ul. \'Sniadeckich 8,  00-656 Warszawa, Poland}
\email{\texttt{piotr.koszmider@impan.pl}}

\thanks{{The research of the second author was supported  by a PAPIIT grant IN100317 and CONACyT grant 285130.}}
\maketitle

\begin{abstract} An almost disjoint family $\A$ of subsets of $\N$ is said to be $\R$-embeddable if
there is a function $f:\N\rightarrow \R$ such that the sets  $f[A]$ 
 are ranges of real sequences converging to distinct reals for distinct $A\in \A$. It is well known
 that almost disjoint families which have few separations, such as Luzin families, are not $\R$-embeddable.
We study extraction principles related to  $\R$-embeddability and separation
properties of almost disjoint families of $\N$ as well as their limitations. An extraction principle whose consistency  is our main result
is:
\begin{itemize}
\item every almost disjoint family of size continuum contains an $\mathbb R$-embeddable subfamily of size continuum.
\end{itemize}
It is true in the Sacks model. The Cohen model serves to show that the above principle does not follow from the fact that every almost disjoint family of size continuum has two  separated subfamilies of size continuum.
 We also construct in
{\sf ZFC} an almost disjoint family, where no two uncountable subfamilies can be separated but
always a countable subfamily can be separated from any disjoint subfamily.

Using a refinement of the $\R$-embeddability property called a controlled $\R$-embedding
property we obtain the following results concerning Akemann-Doner C*-algebras which are
induced by uncountable almost disjoint families:
\begin{itemize}
\item In {\sf ZFC} there are Akemann-Doner C*-algebras of density  $\mathfrak c$ with no commutative subalgebras
of density  $\mathfrak c$,
\item It is independent from {\sf ZFC} whether there is an Akemann-Doner algebra of density $\mathfrak c$ with no nonseparable commutative subalgebra.
\end{itemize}
This completes an earlier result  that there is in ZFC an Akemann-Doner algebra
 of density $\omega_1$ with no nonseparable commutative subalgebra.
 \end{abstract} 

\section{ Introduction}

A family $\mathcal A$ of infinite subsets of $\N$ is \emph{almost disjoint} if any two distinct elements of $\mathcal A$ have finite intersection. 
The earliest uncountable almost disjoint families considered by Sierpi\'nski were
defined as the ranges of sequences of rationals converging to distinct reals.
Hence, we say that an almost disjoint family $\A$ is \emph{$\R$-embeddable} if there is
a  function (called an \emph{embedding}) $f:\N\rightarrow \R$ such that the sets $f[A]$ for $A\in \A$ are
the ranges of sequences converging to distinct reals (see e.g.\cite{guzman-hrusak, hernandez-hrusak}). 
Two families $\mathcal{B},\mathcal{C}$ of subsets of $\N$
 are \emph{separated }if there is $X\subseteq \N$ such that:
\begin{enumerate}
\item If $B\in\mathcal{B}$ then $B\setminus X$ is finite.
\item If $C\in\mathcal{C}$ then $C\cap X$ is finite.\qquad
\end{enumerate}
Considering disjoint neighbourhoods of two condensation points
of  the limits of converging sequences we see that $\R$-embeddable almost disjoint families contain many pairs of 
uncountable subfamilies which are separated. On the other hand 
it is an old and beautiful result of Luzin (\cite{luzin}) that 
there is an almost disjoint family $\mathcal A$
of size $\omega_1$ such that no two  uncountable subfamilies of $\mathcal A$ can be separated.
We will call such families \emph{inseparable}. To highlight the relationship between inseparable 
and $\mathbb R$-embeddable families, recall a dichotomy of \cite{guzman-hrusak} where it is shown that
assuming the proper forcing axiom ({\sf PFA}) every almost disjoint family 
of size $\omega_1$ is either $\mathbb R$-embeddable
 or contains an inseparable subfamily, while Dow \cite{dow} showed that under the
 same assumption every maximal almost disjoint family  contains an inseparable subfamily.

 An uncountable almost disjoint family $\A$ is called a \emph{Q-family} if for every $\B\subseteq \A$
the families $\B$ and $\A\setminus \B$ are separated (sometimes called a
\emph{separated} family). One of the earliest
applications of Martin's axiom  ({\sf MA}) was proving the consistency of 
 the existence of Q-families (which is false under the continuum hypothesis ({\sf CH}) by a counting argument).
All Q-families are $\R$-embeddable and moreover they have a stronger
uniformization type property: for every $\phi:\A\rightarrow \R$ there is $f:\N\rightarrow \R$
such that $f[A]$ is the range of a sequence converging to $\phi(A)$ (in other
words $\lim_{n\in A}(f(n)-\phi(A))=0$) for each $A\in \A$
(\cite[Propositions 2.1., 2.3]{hernandez-hrusak}).

It is natural, and useful (see e.g., \cite[Theorem 2.39]{injective}), to consider  
versions
of the above notions which are more cardinal specific: Let $\kappa$ be a cardinal, then 
\begin{itemize}
\item an almost disjoint family $\A$ has the
\emph{$\kappa$-controlled $\R$-embedding property} if for every   $\phi:\A\rightarrow \R$ there is 
$\B\subseteq \A$ of cardinality $\kappa$ and $f:\N\rightarrow \R$
such that $f[B]$ is the range of a sequence converging to $\phi(B)$ for every $B\in\B$,
\item   an almost disjoint family $\mathcal A$ of size $\kappa$  is \emph{$\kappa$-inseparable}
 if no two subfamilies  of $\mathcal A$ both of size $\kappa$ can be separated, 
\item   an almost disjoint family $\mathcal{A}$ is \emph{$\kappa$-anti Lusin} if it has cardinality 
$\kappa$ and for every subfamily $\mathcal{B}\subseteq A$ of cardinality $\kappa$
there are two subfamilies
$\mathcal{B}_{0},\mathcal{B}_{1}\subseteq \mathcal{B}$ of cardinality $\kappa$ which
can be separated (\cite{roitman-soukup}).
\end{itemize}
\smallskip

This paper is a contribution to the study of extraction principles for
almost disjoint families in the context of the above properties. Our main positive results 
concern the cardinality of the continuum $\mathfrak c$ and are:
\begin{itemize}
\item  It is consistent that  every almost disjoint family of size $\mathfrak c$ contains an
 $\mathbb R$-embeddable subfamily of size $\mathfrak c$ (Theorem \ref{sacks-main}).
\item  It is consistent that  every almost disjoint family of size $\mathfrak c$ has 
the $\omega_1$-controlled $\R$-embedding property (Theorem \ref{controlled-main}).
\item The above extraction principles are not  consequences of  
every  almost disjoint family of size $\mathfrak c$ containing a $\mathfrak c$-anti Luzin subfamily
(Theorems \ref{cohen-dow-hart} and \ref{cohen-main}).
\end{itemize}

The first two  extraction principles above are obtained in the iterated Sacks model.
As a side product we also prove  that in that model every partial function $f:X\rightarrow 2^\N$ for
$X\subseteq2^\N$ of cardinality $\mathfrak c$ is uniformly continuous on an
uncountable $Y\subseteq X$ (Theorem \ref{reversing}). We do not know
if the consistency of this property of functions can be concluded from
known results like in \cite{shelah-meager} or \cite{pawlikowski} or the fact that
under {\sf PFA} every function is monotone on an uncountable set (see \cite{baumgartner}).

The third result above is obtained in the Cohen model from  a  result of
 Dow and Hart (Theorem \ref{cohen-dow-hart}) stating that in that model every almost disjoint family is 
 $\mathfrak c$-anti Luzin (\cite[Proposition 2.6.]{dow-hart} using
 Stepr\=ans's characterization of $\mathcal P(\N)/Fin$ in that model (\cite{steprans})
and from the first of our negative results below:

\begin{itemize}
\item  In the Cohen model there is an almost disjoint family  of cardinality $\mathfrak c$
 with no uncountable $\R$-embeddable subfamily (Theorem \ref{cohen-main}).
\item  In the Cohen model no uncountable almost disjoint family has $\omega_1$-controlled $\R$-embedding
property (Theorem \ref{cohen-no-controlled}).
\end{itemize}

We should recall here that  by a result of A. Avil\'es, F. Cabello S\'anchez,  J. Castillo, M.  
Gonz\'alez and  Y. Moreno  it is consitent (follows from {\sf MA}) that $\mathfrak c$-inseparable
 families exist (\cite[Lemma 2.36]{injective}) ($\mathfrak c$-inseparable
 families are called $\mathfrak c$-Lusin families in \cite{dow-hart, injective}).

On the other hand, we also discover some {\sf ZFC} limitations to other extraction principles:
\begin{itemize}
\item No  almost disjoint family of size $\mathfrak c$ has  the $\mathfrak c$-controlled embedding property (Theorem \ref{no-c-controlled}).
\item There is in {\sf ZFC} an inseparable family of cardinality $\omega_1$ which
 has all possible separations (i.e., separating its countable parts
  from the rest of the family) (Corollary \ref{superlusin}).
\end{itemize}

 The second result is not only natural in the above context by showing that
one cannot even consistently hope for extracting from every inseparable family an uncountable subfamily
with even fewer separations (for example like Mr\'owka's family where one can only separate finite subfamilies
from the rest of the family). It has also found a natural application in a construction 
of a thin-tall scattered operator algebra in \cite{thin-tall}. Note that under
the hypothesis of $\mathfrak b>\omega_1$ all inseparable families have the properties
of our family from Corollary \ref{superlusin} (see \cite[Theorem 3.3]{vandouwen}).
\vskip 13pt
Some of the above results concerning the $\R$-embeddability of almost disjoint families find immediate
applications in the theory of C*-algebras. It was in the paper \cite{akemann-doner}
of Akemann and Doner where certain C*-algebras were associated to an almost disjoint
family $\A$ and a function $\phi: \A\rightarrow[0,2\pi)$. We call these algebras Akemann-Doner
algebras and denote them by $AD(\A, \phi)$. For the construction see Section 6 or the papers 
\cite{akemann-doner, tristan-piotr}. These algebras, initially for $\A$ and $\phi$ constructed only under
{\sf CH} in \cite{akemann-doner}, were the first examples providing negative answer to a question
of Dixmier whether every nonseparable C*-algebra must contain a nonseparable commutative C*-subalgebra.
Later S. Popa found in \cite{popa} a different  and a {\sf ZFC} example, the reduced group C*-algebra
of an uncountable free group. However, the latter C*-algebra is very complicated (e.g. it has no nontrivial idempotents \cite{voiculescu} etc.) while Akemann-Doner algebras are approximately finite dimensional in the sense of \cite{farah-katsura} that is, there is 
a directed family of finite-dimensional C*-subalgebras whose union is dense in the
entire C*-algebra. In \cite{tristan-piotr} it was noted that employing
an inseparable family $\A$ one can obtain in {\sf ZFC} a nonseparable  Akemann-Doner 
algebra with no nonseparable commutative subalgebra. Such {\sf ZFC} examples must be 
obtained from almost disjoint families $\A$ of cardinality $\omega_1$. This
is because we have, for example, the above mentioned
result of Dow and Hart that it is consistent that every almost disjoint family of
cardinality $\mathfrak c$ is $\mathfrak c$-anti-Lusin. The cardinality of the
almost disjoint family $\A$ is the density of the C*-algebra $AD(\A, \phi)$, that is minimal cardinality
of a norm-dense set.
Some natural questions remained, for example,
if one can have in {\sf ZFC} an Akemann-Doner algebra of density $\mathfrak c$ with no nonseparable
commutative subalgebra or another question  if it is consistent that every Akemann-Doner algebra of density $\mathfrak c$
has a commutative C*-subalgebra of density $\mathfrak c$. Here we answer these question
proving that:

\begin{itemize}
\item In {\sf ZFC} there are Akemann-Doner C*-algebras of density  $\mathfrak c$ with no commutative subalgebras
of density  $\mathfrak c$  (Theorem \ref{no-c-com-zfc}).
\item It is independent from {\sf ZFC} whether there is an Akemann-Doner algebra of density $\mathfrak c$ with no nonseparable commutative subalgebra (Theorem \ref{nonsep-com-sacks} and the result of \cite{akemann-doner}).
\end{itemize}

In fact, we also prove in Theorems \ref{no-nonsep-com-cohen} and \ref{superno-nonsep-com-cohen} 
that the existence of nonseparable commutative C*-subalgebras
in every Akemann-Doner algebra  does not follow from the negation of {\sf CH}.

The structure of the paper is as follows: in Section 2 we prove some preliminary {\sf ZFC} results
concerning $\R$-embeddability, Section 3 is devoted to the construction of an inseparable almost disjoint family where all countable parts can be separated from the remaining part of the family, Section 4  is devoted to the results mentioned above that hold in the Cohen model and Section 5 to the results that hold in the Sacks  model. The last section 6 concerns the consequences of the previous results for
the Akemann-Doner C*-algebras.

The set-theoretic terminology and notation is standard and can be found in \cite{jech}.
The knowledge on C*-algebras required to follow Section 6 does not exceed a linear algebra course
concerning $2\times 2$ matrices. Any additional background can be found in \cite{murphy}.

All almost disjoint families are assumed to  be infinite and consist of infinite sets. 
$A\subseteq^*B$ means that
$B\setminus A$ is finite. We use $\N$, $\R$, $\Q$ for nonnegative integers, reals and rationals respectively. When we view elements of $\N$ as von Neumann ordinals, i.e. subsets and/or elements of each other then we use
$\omega$ for $\N$. The cardinality of $\R$ is denoted by $\mathfrak c$.
 If $\kappa$ is a cardinal and $X$ is a  set, then  $[X]^\kappa$ denotes
the family of all subsets of $X$ of cardinality $\kappa$. In particular $[A]^2$ is the set
 of all pairs $\{a,b\}$  of elements of $A$. Elements of $A^n$ for $n\in \omega$ are 
$n$-tuples of $A$ i.e.,  $t= (t(0), t(2),\dots, t(n-1))$.  We consider
$2^{<\omega}=\bigcup_{n\in \omega}2^n$ with the inclusion as a tree,
we also consider its subtrees $T$ and then   $[T]$ denotes the set of all branches of $T$.
The terminology concerning the Cohen forcing $\C$ and the Sacks forcing $\SSS$ is recalled 
at the beginning of Sections 4 and 5 respectively.

\section{Preliminaries}

\subsection{$\R$-embedability of almost disjoint families}

Recall the definition of an $\R$-embeddable  almost disjoint family  from the introduction.
A useful tool for describing properties of almost disjoint families are $\Psi$-spaces associated with them (\cite{michael-ad}).
The $\Psi$-space corresponding to an almost disjoint family $\A\subseteq \wp(\N)$ whose 
points are identified with $\N\cup\A$ is denoted by $\Psi(\A)$.

\begin{lemma}\label{psi-correspondence} Suppose that $\A$ is an almost disjoint family.
There is a 1-1 correspondence between continuous functions $\phi: \Psi(\A)\rightarrow \R$
and functions $f:\N\rightarrow \R$ for which $x_A=\lim_{n\in A}f(n)$ exists for
each $A\in \A$. It is given by $f=\phi\upharpoonright\N$. Then $x_A=\phi(A)$ for
each $A\in \A$.
\end{lemma}

\begin{lemma}\label{equivalences} Let $\A\subseteq \wp(\N)$ be an  almost disjoint family.
Consider $\N^{<\omega}\cup \N^\omega$ with the topology where $\N^{<\omega}$ is
discrete and the basic neighbourhoods of  $x\in \N^\omega$ are of
the form
$$\{y\in \N^{<\omega}\cup \N^\omega\mid y(n)=x(n) \ \hbox{for all}\  n\in F\},$$
for any finite $F\subseteq\omega$.
The following conditions are equivalent (to the property of being $\R$-embeddable):
\begin{enumerate}

\item There is a continuous  $\phi: \Psi(\A)\rightarrow \R$ such that
 $\phi\upharpoonright\A$ is injective,

\item There is a continuous  $\phi: \Psi(\A)\rightarrow \R$ such that
 $\phi\upharpoonright\A$ is injective and $\phi[\A]$ has dense complement in $\R$,
 
 \item There is a continuous  $\phi: \Psi(\A)\rightarrow \R$ such that
 $\phi\upharpoonright\A$ is injective and $\phi[\A]\subseteq \R\setminus\Q$,

\item There is a continuous  $\phi: \Psi(\A)\rightarrow \R$ such that
 $\phi$ is injective,  $\phi[\A]\subseteq \R\setminus\Q$ and $\phi[\N]\subseteq \Q$,

\item There is a continuous  $\phi: \Psi(\A)\rightarrow \N^{<\omega}\cup \N^\omega$ such that
 $\phi$ is injective,  $\phi[\A]\subseteq \N^\omega$ and $\phi[\N]\subseteq \N^{<\omega}$,
 
 \item There is a continuous  $\phi: \Psi(\A)\rightarrow 2^\omega$ such that
 $\phi\upharpoonright\A$ is injective,

 \end{enumerate}
\end{lemma}

\begin{proof} 

(1) $\Rightarrow$ (2)  We may assume that $\A$ is infinite.
Let $U\subseteq \R$ be the union of all open intervals included in $\phi[\A]$.
If it is empty, we are done. Otherwise
let $E=\{e_n\mid n\in \N\}\subseteq U$ be countable and dense in $U$.
 A continuous $\phi':\Psi(\A)\rightarrow \R\setminus E$
such that $\phi'[\Psi(\A)]\subseteq \phi[\Psi(\A)]$ will satisfy (2). Let $\{x^n_k\mid n, k\in \N\}\subseteq \phi[\A]$ be distinct
where $x^n_0=e_n$ for each $n\in \N$
and such that $|x^n_k-x^n_{k+1}|<1/(n+k)$.
 We may choose such $x_n^k$s since $e_n$'s are in the interior of $\phi[\A]$.
Let $A^n_k\in \A$ be such that $\phi(A^n_k)=x^n_k$ for each $n, k\in \N$. 
Find finite $G^n_k\subseteq A^n_k$ so that $A^n_k\setminus G^n_k$'s are all pairwise disjoint and 
$|\phi(i)-x^n_k|<1/(n+k)$ for each $i\in A^n_k\setminus G^n_k$ for each $n, k\in \N$.

Modify $\phi$ to obtain $\phi'$
in the following way: Put $\phi'\upharpoonright
A^n_k\setminus G^n_k$ to be constantly $x^n_{k+1}$ for each
$n, k\in \N$ and $\phi'(A^n_k)=x^n_{k+1}$ for each
$n, k\in \N$ and put $\phi'$ to be equal to $\phi$ on the remaining points of $\Psi(\A)$.

Injectivity of $\phi'\upharpoonright \A$ and the inclusion 
$\phi'[\Psi(\A)]\subseteq \phi[\Psi(\A)]\setminus E$ are clear. 
So we are left with the continuity. $\phi'$ is clearly continuous at each $A^n_k$ for $n, k\in \N$.
Let $A\in \A$ be distinct than each $A^n_k$. Then each intersection $A\cap A^n_k$ is finite.
As $|\phi'(i)-\phi(i)|<2/(n+k)$ for $i\in A^n_k$ for each $n, k\in N$, it follows that
$\lim_{i\in A}|\phi'(i)-\phi(i)|=0$, that is 
$$\lim_{i\in A}\phi'(i)=\lim_{i\in A}\phi(i)=\phi(A)=\phi'(A).$$

(2) $\Rightarrow$ (3)  Choose dense countable $E\subseteq \R\setminus \phi[\A]$. 
Let $\eta:\R\rightarrow \R$ be a homeomorphism such that $\eta[E]=\Q$ and consider
$\phi'=\eta\circ\phi$.

(3) $\Rightarrow$ (4)  Take $\phi$ satisfying (3) and modify it on $\N$ to obtain $\phi'$
in such a way that $\phi'(n)$s are distinct rationals for all $n\in \N$ and 
$|\phi(n)-\phi'(n)|<1/n$ for all $n\in\N$.

(4) $\Leftrightarrow$ (5) First we construct certain bijection 
$\rho: \N^{<\omega}\cup \N^\omega\rightarrow \R$ such that $\rho[\N^{<\omega}]=\Q$
and $\rho[\N^\omega]=\R\setminus\Q$. First define a family
$(I_s\mid s\in\N^{<\omega})$ of open intervals with rational end-points with the following
properties:
\begin{enumerate}
\item $I_\emptyset=\R$,
\item $\bigcup\{{\overline{I_{s^\frown n}}}\mid n\in \N\}=I_s$,
\item each end-point of an interval $I_s$ is an endpoint of another
interval $I_{s'}$ for $|s|=|s'|$,
\item the diameter of $I_s$ is smaller than $1/|s|$ for $s\not=\emptyset$,
\item for every $s\in \N^{<\omega}$ we have $I_{s^\frown n}\cap I_{s^\frown n'}=\emptyset$
for distinct $n, n'\in \N$,
\item Every rational is used as an end-point of two (and necessarily only two adjacent, by the previous properties) of the intervals $I_s$ for $s\in \N^{<\omega}$. $0$ is end-point of
two of $I_s$s for some $|s|=1$.
\end{enumerate}
First define $\rho$ on $\N^{<\omega}$ by
defining $\rho(s)$ by induction on  $|s|$. Let $\rho(\emptyset)=0$.
If $|s|=1$, then $\rho(s)$ is the right end-point of $I_s$ if $I_s$ consists of positive reals
and $\rho(s)$ is the left end-point of $I_s$ if $I_s$ consists of negative reals.
If $|s|>1$, then $\rho(s)$ is the left end-point of $I_s$.
For $x\in \N^\omega$ let $\rho(x)$ be the only point of
$\bigcap_{n\in \omega}I_{x\upharpoonright n}$.

Note that $\rho$ is continuous and that $\rho^{-1}(x_n)\rightarrow \rho^{-1}(x)$
if $x_n$ is a sequence of rationals converging to an irrational $x$. This
proves (4) $\Leftrightarrow$ (5).

(5) $\Rightarrow$ (6) First note that there  $\eta:  \N^{<\omega}\cup \N^\omega\rightarrow \N^\omega$
which is continuous and the identity on $\N^\omega$. Namely send $s\in \N^{<\omega}$
to the sequence $s^\frown0^\omega$. Now note that there is $\zeta:\N^\omega\rightarrow 2^\omega$
which is continuous. So use the composition of these functions to obtain (6) from (5).

(6) $\Rightarrow$ (1)  is clear.
\end{proof}

\begin{remark}\label{remark-embedding} Using Lemma \ref{psi-correspondence} we obtain versions of the conditions
from Lemma \ref{equivalences} for functions from $\N$ into $\R$. In particular the definition
of an $\R$-embeddable almost disjoint from the introduction which is
 a version of item (1) of Lemma \ref{equivalences}  is equivalent 
 to version in the literature, e.g. in \cite{hernandez-hrusak} which
are versions of item (4) of Lemma \ref{equivalences}.
\end{remark}

The following is a simple condition that allows us to get $\R$-embeddability.

\begin{lemma}\label{tree-condition}
Let $T\subseteq2^{<\omega}$ be a tree, $Z\subseteq\left[  T\right]  $ and
$\mathcal{A}=\left\{  A_{r}\mid r\in Z\right\}  $ an almost disjoint  family of
subsets of $\N$. If
there is a family $\left\{  B_{s}\mid s\in T\right\}  \subseteq\left[
\N\right]  ^{\omega}$ with the following properties:

\begin{enumerate}
\item   $B_{t}=\bigcup\{B_{t^\frown i}\mid t^\frown i\in T, \ i\in \{0,1\}\}$ for all $t\in T$,
\item $B_{s}\cap B_{t}$ is finite whenever $s, t \in T$ are incompatible.
\item $A_{r}\subseteq\bigcap_{n\in\omega}
B_{r\upharpoonright n}$ for every $r\in Z.$
\end{enumerate}
Then, $\mathcal{A}$ is $\R$-embeddable.
\end{lemma}

\begin{proof}

Define  $\phi:\Psi(\A)\longrightarrow 2^\omega$ by puting $\phi(A_r)=r$ for all $r\in Z$
and $\phi(n)=s^\frown0^\omega$ if $n\in B_s$, $|s|\geq n$ and $s$ is the first in the lexicographic
order which satisfies the previous requirements. If there is no such $s\in 2^{<\omega}$, then put
$\phi(n)=0^\omega$.  Clearly $\phi\upharpoonright \A$ is injective, so we are left with the continuity to
check (1) of Lemma \ref{equivalences}.

 By (3) if $k\in A_r$, then $k\in B_{r\upharpoonright n}$ for
every $n\in \omega$. Fix $n\in\omega$. So if we take 
$$k\in A_r\setminus \bigcup\{B_t\mid |t|=n, \ t\not=r|n     \},\leqno (*)$$
then the condition `` $k\in B_s$ and $|s|\geq n$" implies 
$r\upharpoonright n\subseteq s$ by (1). By (2) the set
in (*) almost covers $A_r$, and so for almost all elements of $k\in A_r$ we have
$r\upharpoonright n\subseteq \phi(k)$. As $n\in\omega$ was arbitrary, it follows that $\lim_{k\in A_r}\phi(k)=r=\phi(A_r)$, as required for the continuity.
\end{proof}

\begin{remark} By transfinite induction one can construct a  family of sequences 
$(q^\alpha_n)_{n\in \N}$ for  $\alpha<\mathfrak c$  in such a way that 
no tree $T\subseteq 2^{<\omega}$ and no collection $\{B_t\mid t\in T\}$  satisfies the
hypothesis of Lemma \ref{tree-condition} for any family of $\wp(\N)$ obtained through a bijection
between $\N$ and $\Q$ from $\{\{q^\alpha_n\}_{n\in \N}\mid \alpha<\mathfrak c\}$.
 It follows that the  condition from Lemma \ref{tree-condition}
is not equivalent to the $\R$-embeddability. 
This way one can also conclude that  there are $\R$-embeddable almost disjoint 
families of subsets of $\N$
which are not equivalent to a family of branches of $2^{<\omega}$.
\end{remark}

\subsection{$\kappa$-controlled $\R$-embedding property}

Recall the definition of the $\kappa$-controlled $\R$-embedding property from the introduction.

\begin{theorem}\label{no-c-controlled} No almost disjoint family $\A$ of cardinality $\mathfrak c$ has
$\mathfrak c$-controlled $\R$-embedding property.
\end{theorem}
\begin{proof} Let $\A$ be an almost disjoint family of size $\mathfrak c$
consisting of infinite sets.
Let $(M_\alpha)_{\alpha<\mathfrak c}$ be a well-ordered, continuous, increasing chain of sets satisfying
\begin{enumerate}
\item $|M_\alpha|\leq\max(|\alpha|, \omega)$ for each $\alpha<\mathfrak c$,
\item $\R^\N, \wp(\N)\subseteq\bigcup_{\alpha<\mathfrak c}M_\alpha$,
\item If $A\in M_\alpha\cap \wp(\N)$ and $f\in M_\alpha\cap\R^\N$ and 
 $\lim_{n\in A}f(n)$ exists, then it belongs to $M_{\alpha+1}$.
\end{enumerate}
It should be clear that one can construct such a sequence $(M_\alpha)_{\alpha<\mathfrak c}$.
Define $\phi:\A\rightarrow [0,1]$ so that $\phi(A)\in\R\setminus M_{\alpha(A)+1}$ for
$A\in \A$, where
$$\alpha(A)=\min\{\alpha<\mathfrak c\mid A\in M_\alpha\}.$$

This can be arranged by (2) and by (1). Now suppose $\A'\subseteq \A$ has cardinality 
$\mathfrak c$ and $f:\N\rightarrow \R$.
By (2) there is $\alpha_0<\mathfrak c$ such that $f\in M_{\alpha_0}$. Take $A\in \A'$ such that
$\alpha(A)\geq\alpha_0$ which exists by (1) as $\A'$ has cardinality $\mathfrak c$.
 Then $A\in M_{\alpha(A)}\cap \wp(\N)$ and $f\in M_{\alpha(A)}\cap[0,1]^\N$, so by (3), 
 if  $\lim_{n\in A}f(n)$ exists, then it belongs to $M_{\alpha(A)+1}$. But $\phi(A)\not\in M_{\alpha(A)+1}$
 by the definition of $\phi$, so $\lim_{n\in A}f(n)\not=\phi(A)$.
\end{proof}

However, it is quite possible to have almost disjoint families
of cardinality $\kappa$ with $\kappa$-controlled embedding property:

\begin{proposition}\cite[cf. 2.3.]{hernandez-hrusak}  
Let $\kappa$ be a cardinal. Assume {\sf MA}$_\kappa$.  Then
every  subfamily $\A$  of cardinality $\kappa$ of the Cantor family
 $\mathcal C=\{A_x\mid x\in 2^\omega\}\subseteq \wp(2^{<\omega})$, where 
 $A_x=\{x\upharpoonright n\mid n\in \omega\}$
  for $x\in 2^\omega$,
has the following strong version of  the $\kappa$-controlled embedding property:
For every function $\phi:\mathcal A\rightarrow [0,1]$ there is  a 
function $f: 2^{<\omega}\rightarrow [0,1]$ such that for all $A\in \A$ 
$$\lim_{s\in A}f(s)=\phi(F).$$
\end{proposition}
\begin{proof} 
It is well known that under the above hypothesis all subsets
of $2^\omega$ of cardinality $\kappa$ are $Q$-sets and that it implies that
all subfamilies of the Cantor family of cardinality $\kappa$ 
can be separated from the rest of the family, i.e. they are $Q$-families in our terminology
from the introduction. It follows that
$\Psi(\A)$ is a normal topological space. As the nonisolated points of
$\Psi(\A)$ correspond to $\A$ and form a discrete closed subset of 
$\Psi(\A)$ any function $\phi$ on them is continuous and extends  by the Tietze extension theorem to a continuous
$\widetilde\phi: \Psi(\A)\rightarrow [0,1]$. So put $f=\widetilde\phi\upharpoonright 2^{<\omega}$
and use Lemma \ref{psi-correspondence} identifying $2^{<\omega}$ and $\N$.
\end{proof}

\section{A Luzin family with all possible separations in ZFC}

The main striking property of a Luzin family is that it is inseparable.
On the other hand,
there is also an almost disjoint  family $\mathcal A$  of size $\aleph_1$ such that every
 countable $\mathcal B\subseteq \mathcal A$ can be separated from $\mathcal A\setminus\mathcal B$ (see \cite{lambda}). 
Here we construct an almost disjoint family which 
satisfies both properties simultaneously. As both of these properties are hereditary with respect to
uncountable subfamilies this shows certain limitations to any further extraction principles.

To construct the almost disjoint family with the aboved-mentioned properties
we need colorings of pairs of countable ordinals with properties 
first obtained by S. Todorcevic  in \cite{todorcevic-acta} (cf. \cite{todorcevic-walks}).
In fact, the concrete construction we choose, due to Velleman (\cite{velleman}), is based
on a family of finite subsets of $\omega_1$. It was C. Morgan
who connected these two ideas (\cite{morgan}).
For functions $c:[\omega_1]^2\rightarrow \N$ we will abuse notation and
denote $c(\{\alpha, \beta\})$ by $c(\alpha, \beta)$.

\begin{theorem}\label{rho} There is a sequence $(g_\alpha\mid \alpha<\omega_1)\subseteq \{0,1,2\}^\N$
and a coloring $c:[\omega_1]^2\rightarrow\N$
satisfying the following:
\begin{enumerate}

\item For all $\beta<\alpha<\omega_1$ for all $k>c(\beta, \alpha)$
we have $\{g_\beta(k), g_\alpha(k)\}\not=\{1, 2\}$,
\item For all $\beta<\alpha<\omega_1$ 
we have $g_\beta(c(\beta, \alpha))=1$ and $g_\alpha(c(\beta, \alpha))\}=2$,
\item For all $\gamma<\beta<\alpha<\omega_1$ if  $c(\gamma, \beta)>c(\alpha, \beta)$, then 
$c(\gamma, \beta)=c(\gamma, \alpha)$,
\item For all $\alpha<\omega_1$ and all $m\in \N$  the set $\{\beta<\alpha\mid c(\beta, \alpha)<m\}$ is finite.
\item For all $\alpha<\omega_1$ the sets and $g_\alpha^{-1}[\{1\}]$ and $g_\alpha^{-1}[\{2\}]$
are infinite.
\end{enumerate}
\end{theorem}
\begin{proof} We choose the approach from Section 5 of \cite{2-cardinals}.
Thus our $c:[\omega_1]^2\rightarrow \N$ is $m$ of Definition 5.1. of \cite{2-cardinals}, i.e., $c(\alpha, \beta)$ is
the minimal rank of an element $X\in \mu$ such that $\alpha, \beta\in X$ where $\mu$ is
an $(\omega, \omega_1)$-cardinal. 

The functions $g_\alpha$ for $\alpha<\omega_1$ are defined as follows, for $n=0$ we put $g_\alpha(0)=0$
for any $\alpha<\omega_1$ and for any $n\in \N$ we put: 

\[g_\alpha(n+1)=\begin{cases}
0&\text{if } \exists X_1*X_2\in \mu\ rank(X_1)=rank(X_2)=n, \ \alpha\in X_1\cap X_2,\\
1&\text{if } \exists X_1*X_2\in \mu\ rank(X_1)=rank(X_2)=n, \ \alpha\in X_1\setminus X_2, \\
2&\text{if } \exists X_1*X_2\in \mu\ rank(X_1)=rank(X_2)=n, \ \alpha\in X_2\setminus X_1. \\
\end{cases}\]

Here $X_1*X_2$ is as in the definition 1.1. (5) of \cite{2-cardinals}.
 First let us argue that the  $g_\alpha$s are well defined.
By Definition 1.1. (6) and (7) of \cite{2-cardinals} each element $\alpha\in \omega_1$
is in an element of rank zero of $(\omega, \omega_1)$-cardinal $\mu$. Now by 
Velleman's Density Lemma 2.3.  of \cite{2-cardinals} it follows that
$\alpha$ is in an element of rank $n$ of $\mu$ for any $n\in \N$.
By Definition 1.1. (5) of \cite{2-cardinals} each element $X$ of $\mu$ 
of rank bigger than zero is of the form
$X_1*X_2$ which means in particular that
 $X=X_1\cup X_2$ and $X_1\cap X_2<X_1\setminus X_2<X_2\setminus X_1$.
Now suppose that $\alpha\in X=X_1*X_2$ and $\alpha\in Y= Y_1*Y_2$ and the ranks of $X_1, X_2, Y_1, Y_2$
are elements of $\mu$ of fixed rank $n\in \N$.  By Definition 1.1. (3) of \cite{2-cardinals}
there is an order preserving $f_{Y, X}: X\rightarrow Y$, which by By Definition 1.1. (3)  and (5) of
 \cite{2-cardinals} must satisfy $f[X_1]=Y_1$ and $f[X_2]=Y_2$ and moreover
$f\upharpoonright(X\cap(\alpha+1))$ is the identity on $X\cap(\alpha+1)$ be the coherence
 lemma 2.1 of \cite{2-cardinals}, so $f_{Y, X}(\alpha)=\alpha$ and $f[X_1\cap X_2]=Y_1\cap Y_2$,
$f[X_1\setminus X_2]=Y_1\setminus Y_2$ and $f[X_2\setminus X_1]=Y_2\setminus Y_1$ and so
the value of $g_\alpha(n+1)$ does not depend if we applied the definition of $g_\alpha(n+1)$
to $X_1*X_2$ or $Y_1*Y_2$ which completes the proof of the claim that the $g_\alpha$s are well defined.

Now we will  prove (1)  and (2) for $\alpha<\beta<\omega_1$
such that $c(\alpha, \beta)>0$.
For (1) let $n+1=k>rank(X)$ such that $\alpha, \beta\in X\in \mu$.  Let $Y$ (which exists by
the above-mentioned Density Lemma) be such that $X\subseteq Y\in \mu$ and $rank(Y)=k$. 
$Y=Y_1*Y_2$. By Definition 1.1. (5) of \cite{2-cardinals} we have that $X\subseteq Y_1$ or
$X\subseteq Y_2$, so $\{g_\beta(k), g_\alpha(k)\}\not=\{1, 2\}$.
(2) follows from the definition of $c$,  i.e., from the minimality of the rank of $X\ni\alpha, \beta$,
which is of the form $X_1\cup X_2$ with $X_1\setminus X_2< X_2\setminus X_1$ by By Definition 1.1. (5) of \cite{2-cardinals}
and by the hypothesis  that $c(\alpha, \beta)>0$.
Property (3) is Corollary 5.4 (2) of \cite{2-cardinals}. 
Property (4) is Proposition 5.3 (a) of \cite{2-cardinals}. 

To obtain property (5),
recall from \cite[Theorem 4.5]{2-cardinals}
that $(g^{-1}_\alpha[\{1\}], g^{-1}_\alpha[\{2\}])_{\alpha<\omega_1}$
is a Hausdorff gap, so the sets must be infinite from some point on, so it is enough
to remove possibly countably many $\alpha<\omega_1$ and renumerate the remaining ones.

So we are left with removing the hypothesis $c(\alpha,\beta)>0$ from (1) and (2).
Note that what we have  proved  so far is valid for $\alpha, \beta, \gamma$ from
any subset of $\omega_1$, in other words we can pass to an uncountable subset $X$ of $\omega_1$
and consider only $g_\alpha$s for $\alpha\in X$ and then re-enumerate $X$ as $\omega_1$ in an
increasing manner. So we need to argue that there is an uncountable $X\subseteq\omega_1$
such that $c(\alpha,\beta)>0$ for every $\alpha<\beta$ and $\alpha, \beta\in X$.
 To obtain $X$ apply the Dushnik-Miller theorem (Theorem 9.7 of \cite{jech})
 to a coloring $d:[\omega_1]^2\rightarrow \{0,1\}$ given by 
 $d(\alpha,\beta)=\min\{1, c(\alpha, \beta)\}$ knowing that 
all elements of rank zero must have fixed finite cardinality.

\end{proof}

\begin{theorem}\label{complicated-ad} There are families 
$(X_\alpha, Y_\alpha, A_\alpha, B_\alpha\mid \alpha<\omega_1)$ of  subsets of $\N$ 
such that 
\begin{enumerate}
\item $X_\alpha=A_\alpha\cup B_\alpha$ is infinite, $A_\alpha\cap B_\alpha=\emptyset$ for all $\alpha<\omega_1$,
\item $X_\beta\cap X_\alpha=^*\emptyset$ for all $\beta<\alpha<\omega_1$,
\item $Y_\beta\subseteq^* Y_\alpha$ for all $\beta<\alpha<\omega_1$,
\item $X_\beta\subseteq^* Y_\alpha$ for all $\beta<\alpha<\omega_1$,
\item $X_\alpha\cap Y_\alpha=\emptyset$ for all $\alpha<\omega_1$,
\item For every $\alpha<\omega_1$ and every $k\in \N$ for all
but finitely many $\beta<\alpha$ there is $l>k$ 
 such that 
$$l\in A_\beta\cap B_\alpha.$$

\end{enumerate}
\end{theorem}
\begin{proof}
Define all the sets as subsets of  $[\{0,1,2\}^{<\omega}]^2$ instead of $\N$.
For $\alpha<\omega_1$ put $X_\alpha=A_\alpha\cup B_\alpha$, where
$$A_\alpha=\{\{g_\alpha\upharpoonright(n+1), s\}\mid
 s\in \{0,1,2\}^{n+1}, \ g_\alpha(n)=1,  s(n)=2\}, \ n\in \N\}.$$
$$B_\alpha=\{\{g_\alpha\upharpoonright(n+1), s\}\mid
 s\in \{0,1,2\}^{n+1}, \ g_\alpha(n)=2, s(n)=1\}, \ n\in \N\}.$$
 So (1) is clear by Theorem \ref{rho} (5). 

  If $\beta<\alpha<\omega_1$ and
 $\{r,s\}\in X_\alpha\cap X_\beta$ and
    $g_\alpha\upharpoonright(n+1)\not=g_\beta\upharpoonright(n+1)$, then
  $\{r,s\}=\{g_\alpha\upharpoonright(n+1), g_\beta\upharpoonright(n+1)\}$
 and $\{r(n), s(n)\}=\{1,2\}$ which means that $n\leq c(\alpha, \beta)$ by
 (1) and (2) of Theorem  \ref{rho}. So we obtain (2).
 
 Note that if $\beta<\alpha<\omega_1$, then
 $\{g_\alpha\upharpoonright c(\alpha,\beta), g_\beta\upharpoonright c(\alpha,\beta)\}\in A_\beta\cap B_\alpha$ 
 by
 (1) of Theorem  \ref{rho}, so we obtain (6). 
 
For $\alpha<\omega_1$ define
$$Y_\alpha=\bigcup_{\beta<\alpha}\bigg(X_\beta\setminus
 \bigcup_{i\leq c(\beta,\alpha)} [\{0,1,2\}^{i+1}]^2\bigg).$$
 If follows that $X_\beta\subseteq Y_\alpha$ if $\beta<\alpha<\omega_1$, so we have (4).
 Also $Y_\alpha\cap X_\alpha=\emptyset$ holds because
 $X_\beta\cap X_\alpha\subseteq \bigcup_{i\leq c(\beta,\alpha)} [\{0,1,2\}^{i+1}]^2$
 by (1) and (2) of Theorem  \ref{rho}.
 
If $\gamma<\beta<\alpha$ we have $c(\gamma, \beta)=c(\gamma, \alpha)$ with the
possible exception for $\gamma<\beta$ in the  set $D(\beta, \alpha)=\{\delta<\beta\mid c(\delta, \beta)\leq c(\beta, \alpha)\}$
by  (3) of Theorem  \ref{rho}. $D(\beta, \alpha)$ is moreover  finite by (4) of Theorem  \ref{rho}.
So almost all summands in the definition of $Y_\beta$ appear literally in the definition
of $Y_\alpha$. The remaining summands of $Y_\beta$ are $X_\gamma\setminus
 \bigcup_{i\leq c(\gamma,\beta)} [\{0,1,2\}^{i}]^2$ for $\gamma\in D(\beta, \alpha)$.
 Each of them is almost equal to a summand of $Y_\alpha$ of the form
 $X_\gamma\setminus
 \bigcup_{i\leq c(\gamma,\alpha)} [\{0,1,2\}^{i}]^2$ for $\gamma\in D(\beta, \alpha)$
 which proves that $Y_\beta\subseteq^* Y_\alpha$ that is we have (3) which completes the proof 
 of the theorem.

\end{proof}

An example of the use of the partition of $X_\alpha$s above into
$A_\alpha$ and $B_\alpha$ is given in the following
proposition which has found an application in \cite{thin-tall}.

\begin{proposition}\label{complicated-ad2} There are families 
$(X_\alpha', Y_\alpha',  \alpha<\omega_1)$ of  subsets of $\N$ 
and bijections $f_\alpha:\N\times \N\rightarrow X_\alpha'$
such that 
\begin{enumerate}
\item $X_\beta'\cap X_\alpha'=^*\emptyset$ for all $\beta<\alpha<\omega_1$,
\item $Y_\beta'\subseteq^* Y_\alpha'$ for all $\beta<\alpha<\omega_1$,
\item $X_\beta'\subseteq^* Y_\alpha'$ for all $\beta<\alpha<\omega_1$,
\item $X_\alpha'\cap Y_\alpha'=\emptyset$ for all $\alpha<\omega_1$,
\item For every $\alpha<\omega_1$ and every $k\in \N$ for all
but finitely many $\beta<\alpha$ there are $m_1< ...<m_k$ and $n_1< ...< n_k$ 
such that 
$$f_\alpha(i,n_j)=f_\beta(j, m_i) $$
for all $1\leq i, j\leq k$.
\end{enumerate}
\end{proposition}
\begin{proof}
Consider a  pairwise disjoint family $\{I_l\mid l\in \N\}$ of finite subsets $\N$
where $I_l=\{l_{i, j}\mid 1\leq i, j\leq l\}\cup\{r_l\}$. 
Define $X_\alpha'=\bigcup\{I_l\mid l\in X_\alpha\}$
and $Y_\alpha'=\bigcup\{I_l\mid l\in Y_\alpha\}$ where $X_\alpha, Y_\alpha$
  satisfy Theorem \ref{complicated-ad}. It is clear that (1) - (4) are satisfied.
  Put $X_\alpha''=\bigcup\{I_l\setminus\{r_l\}\mid l\in X_\alpha\}$.
Now for $\alpha<\omega_1$ let $A_\alpha$ and $B_\alpha$ be
as in  Theorem \ref{complicated-ad} and define 
recursively in $l\in X_\alpha$  for elements
of $I_l\setminus\{r_l\}$ an injection $h_\alpha: X_\alpha''\rightarrow \N\times \N$ 
in such a way that if $l\in A_\alpha$, then there are $m_1< ...<m_l$ such that
$h_\alpha(l_{i, j})=(j, m_i)$ for all $1\leq i, j\leq l$, and  
if $l\in B_\alpha$, then there are $n_1< ...<n_l$ such that
$h_\alpha(l_{i, j})=(i, n_j)$ for all $1\leq i, j\leq l$. Now
use the elements $\{r_l\mid l\in X_\alpha\}$ to extend $h_\alpha$ to
a bijection $h_\alpha': X_\alpha'\rightarrow \N\times \N$ and define $f_\alpha=(h_\alpha')^{-1}$.
Note that
(6) of Theorem \ref{complicated-ad} gives $l>k$ such that $l\in A_\beta\cap B_\alpha$,
and so (5) follows.
\end{proof}

 We may note several interesting properties
of the almost disjoint family $(X_\alpha\mid \alpha<\omega_1)$ from Theorem \ref{complicated-ad}.

\begin{corollary}\label{superlusin} There is an  almost disjoint family $\A$ which is inseparable (Luzin) but
for every countable $\B\subseteq\A$, the families $\B$ and $\A\setminus \B$
can be separated.
\end{corollary}
\begin{proof} As countable almost disjoint families can be separated, it is enough to
separate the  initial fragment $\{X_\beta\mid \beta<\alpha\}$
from the remaining part $\{X_\beta\mid \beta\geq\alpha\}$.
Our family from  Theorem \ref{complicated-ad} of course has such separation $Y_\alpha$, so it is enough
to note that it is inseparable. For this note that Theorem \ref{complicated-ad} (5) implies that
given $\alpha<\omega_1$ and $k\in \N$ for all but finitely many $\beta<\alpha$ we have
$\max(X_\beta\cap X_\alpha)>k$. This condition implies the inseparability of
the family in the standard way as in the case of the Lusin family (cf. \cite{michael-ad}).
\end{proof}

\begin{corollary} There is a Luzin family $(X_\alpha\mid \alpha<\omega_1)$ such that
whenever $\mathcal X\subseteq\omega_1$ is uncountable, councountable, then there is a
a Hausdorff gap  $(A_\alpha^{\mathcal X}, B_\alpha^{\mathcal X})_{\alpha<\omega_1}$
for which  $((X_\alpha\mid \alpha\in \mathcal X), (X_\alpha\mid \alpha\in \omega_1\setminus \mathcal X))$
is its almost disjoint refinement.
\end{corollary}
\begin{proof} Take the families $(X_\alpha\mid \alpha<\omega_1)$ and 
 $(Y_\alpha\mid\alpha<\omega_1)$ from Theorem \ref{complicated-ad}.
Using the nonexistence
of countable
gaps in $\wp(\N)/Fin$ for each $\alpha<\omega_1$ we can  recursively
construct separation $C_\alpha^X$  
of $(X_\beta\mid \beta\in \mathcal X\cap\alpha)$
and $(X_\beta\mid \beta\in \alpha\setminus\mathcal  X)$ i.e., such $C_\alpha^X\subseteq \N$ that
\begin{itemize}
\item $X_\beta\subseteq^* C_\alpha^X$, if $\beta\in \alpha\cap \mathcal X$,
\item $X_\beta\cap C_\alpha^X=^*\emptyset$,
 if $\beta\in \alpha\setminus \mathcal X$. 
 \item $C_\beta^X\cap Y_\beta\subseteq^*C_\alpha^X$, if $\beta<\alpha$,
 \item $(Y_\beta\setminus C_\beta^X)\cap C_\alpha^X=^*\emptyset$, if $\beta<\alpha$.
 \end{itemize}
 
 Putting
$A_\alpha^X=C_\alpha^X\cap Y_\alpha$, $B_\alpha^X=Y_\alpha\setminus C_\alpha^X$
we obtain a Hausdorff gap.
\end{proof}

\section{ $\R$-embeddability  in the Cohen model}

The Cohen forcing $\C$ consists of elements of $\N^{<\omega}$ and is ordered by
reverse inclusion.
By the \emph{Cohen model }we mean the model obtained by adding
$\omega_{2}$-Cohen reals with finite supports to a model of the Generalized
Continuum Hypothesis (\textsf{GCH}). Given $X\subseteq\omega_{2}$ we define
$\mathbb{C}_{X}$ as the forcing adding Cohen reals (with finite supports)
indexed by $X.$ The following lemma is well known:

\begin{lemma}[Continuous reading of names for Cohen forcing] If $\dot{A}$ is a $\mathbb{C}%
$-name for a subset of $\N,$ then there is a pair $\left(  \left\langle
\mathcal{B}_{n}\right\rangle _{n\in\N},F\right)  $ such that

\begin{enumerate}
\item each $\mathcal{B}_{n}\subseteq\N^{<\omega}$ is a maximal antichain.

\item if $s\in\mathcal{B}_{n+1}$ then there is $t\in\mathcal{B}_{n}$ such that
$t\subseteq s.$

\item $F:%
{\textstyle\bigcup\limits_{n\in\N}}
\mathcal{B}_{n}\longrightarrow2.$

\item If $c\in\N^{\omega}$ is  Cohen over $V,$ then $$\dot{A}\left[
r\right]  =\left\{  n\mid\exists m\left(  \left(
r\upharpoonright m\right)  \in\mathcal{B}_{n}\ \& \ F\left(  c\upharpoonright
m\right)  =1\right)  \right\}  .$$
\end{enumerate}
\end{lemma}

Here by $\dot{A}\left[
c\right]$ we denote the evaluation of the name $\dot{A}$ using the generic real $c$. If the 
 conditions (1) - (4) hold, we will say that $\left(  \left\langle
\mathcal{B}_{n}\right\rangle _{n\in\N},F\right)  $ codes $\dot{A}.$

As a warm-up we present a direct proof of a result of Dow and Hart from \cite{dow-hart}
which was obtained there using an ingenious axiomatization of $\wp(\N)/Fin$ in the Cohen model.

\begin{theorem}[\cite{dow-hart}]\label{cohen-dow-hart} In the Cohen model, 
every almost disjoint family of
size $\omega_{2}$ is $\omega_{2}$-anti Lusin.
\end{theorem}

\begin{proof}
It is enough to show that in the Cohen model, every almost disjoint family of size
$\omega_{2}$ contains two subfamilies of size $\omega_{2}$ that are
separated. Let $\mathcal{A}=\{\dot{A}_{\alpha}\mid\alpha\in\omega_{2}\}$ be a
$\mathbb{C}_{\omega_2}$-name for an almost disjoint family. Since $\mathbb{C}_{\omega_2}$
 has the countable chain condition, for every $\alpha\in
\omega_{2},$ we can find $S_{\alpha}\in\left[  \omega_{2}\right]  ^{\omega}$
such that each $\dot{A}_{\alpha}$ is, in fact, a $\mathbb{C}_{S_{\alpha}}$-name.
By {\sf CH} and the $\Delta$-system lemma, (see \cite{kunen} Lemma III.6.15) we can find
$X\in\left[  \omega_{2}\right]  ^{\omega_{2}}$ such that $\left\{  S_{\alpha
}\mid\alpha\in X\right\}  $ forms a $\Delta$-system with root $R\in\left[
\omega_{2}\right]  ^{\omega}.$

\smallskip

We may further assume that the root $R$ is the empty set (if this is not the
case, we simply move to the intermediate model obtained by forcing with
$\mathbb{C}_{R}$). Since $\mathbb{C}_{S_{\alpha}}$ is a forcing notion equivalent to
$\mathbb{C},$ we may assume that for each $\alpha\in X,$ $\dot{A}_{\alpha}$ is
a $\mathbb{C}_{\left\{  \alpha\right\}  }$-name. Since $V$ is a model of {\sf CH}, 
we can find $X_{1}\in\left[  X\right]  ^{\omega_{2}}$
and a pair $\left(  \left\langle \mathcal{B}_{n}\right\rangle _{n\in\omega
},F\right)  $ that codes every $\dot{A}_{\alpha}.$ In other words, each
$\dot{A}_{\alpha}$ is forced to be equal to 
$$\left\{  n\mid\exists m\left(\left(  \dot{c}_{\alpha}\upharpoonright m\right)
  \in\mathcal{B}_{n}\ \& \
F\left(  \dot{c}_{\alpha}\upharpoonright m\right)  =1\right)  \right\}  $$
(where $\dot{c}_{\alpha}$ is the name of the $\alpha^{\text{th}}$-Cohen real). Since
$\mathcal{A}$ is forced to be an almost disjoint family, there are
$s,t\in\N^{<\omega}$ such that:
\begin{enumerate}
\item $s$ and $t$ are  incomparable nodes of the same length,
\item there are no $m,s^{\prime},t^{\prime}$ with the following properties:
\begin{enumerate}
\item $m>\left\vert s\right\vert, \left\vert t\right\vert .$

\item $s^{\prime},t^{\prime}\in\mathcal{B}_{m}.$

\item $s\subseteq s^{\prime},$ $t\subseteq t^{\prime}.$

\item $F\left(  s^{\prime}\right)  =F\left(  t^{\prime}\right)  =1.$
\end{enumerate}
\end{enumerate}

(In fact, every pair of incomparable nodes can be extended to a pair of
nodes satisfying these properties). In $V\left[  G\right]$, define
families $\mathcal{C}_{0}$ and $\mathcal{C}_{1}$ by 
$$\mathcal{C}%
_{0}=\{\dot{A}_{\alpha}\left[  c_{\alpha}\right]  \mid\alpha\in X_{1}\wedge
s\subseteq c_{\alpha}\}$$ and $$\mathcal{C}_{1}=\{\dot{A}_{\alpha}\left[
c_{\alpha}\right]  \mid\alpha\in X_{1}\wedge t\subseteq c_{\alpha}\}.$$ It is
easy to see that both families are of size $\omega_{2}$ and are separated by
$\bigcup \{A\setminus m\mid A\in \mathcal C_0\}$.
\end{proof}

\smallskip

A stronger statement: \textquotedblleft Every almost disjoint family of size
continuum contains an $\R$-embedabble subfamily of size
continuum\textquotedblright\ is consistent but it is false in the Cohen model.
We will prove the latter in the rest of this section and the former in the
next section.

By $\mathbb{T}$ we denote the set of all finite trees $T\subseteq
\N^{<\omega}$ such that all maximal nodes of $T$ have fixed the same height, we
denote this common value by $ht\left(  T\right)  .$ Given a tree
$T\subseteq\N^{<\omega}$ we define $\left[  T\right]  ^{2,=}=\{\left\{
s,t\right\}  \in\left[  T\right]  ^{2}\mid\left\vert s\right\vert =\left\vert
t\right\vert \}.$

\begin{definition}
Define $\mathbb{P}$ as the collection of all triples $p=\left(  T_{p}%
,R_{p},\phi_{p}\right)  $ that satisfy the following properties:

\begin{enumerate}
\item $T_{p}\in\mathbb{T}.$

\item $R_{p}$ $\subseteq\left[  T_{p}\right]  ^{2,=}.$

\item If $\left\{  s,t\right\}  \in R_{p}$ and $\left\{  s^{\prime},t^{\prime
}\right\}  \in\left[  T_{p}\right]  ^{2,=}$ is such that $s\subseteq
s^{\prime}$ and $t\subseteq t^{\prime}$ then $\left\{  s^{\prime},t^{\prime
}\right\}  \in R_{p}.$

\item $\phi_{p}:T_{p}\longrightarrow2.$

\item There is no $\left\{  s,t\right\}  \in R_{p}$ such that $\phi_{p}\left(
s\right)  =\phi_{p}\left(  t\right)  =1.$
\end{enumerate}

Given $p,q\in\mathbb{P}$ we say $p\leq_{\mathbb{P}}q$ if $T_{q}\subseteq
T_{p},$ $R_{q}=R_{p}\cap\left[  T_{q}\right]  ^{2,=},$ $\phi_{q}\subseteq \phi_{p}.$
\end{definition}

Since $\mathbb{P}$ is a countable partial order, it is a forcing notion equivalent to
the Cohen forcing. We define $\dot{\phi}_{gen}$ to be equal to $%
{\textstyle\bigcup}
\{\phi_{p}\mid p\in\dot{\GG}\}$ (where $\dot{\GG}$ is the name for a generic filter
of $\mathbb{P}$). It is easy to see that $\dot{\phi}_{gen}$ is forced to be a
function from $\N^{<\omega}$ to $2.$

\begin{definition}
We define $\UU$ as the set of all sequences $\left(  p,\left\langle
s_{\alpha}\right\rangle _{\alpha\in F}\right)  $ with the following properties:

\begin{enumerate}
\item $p\in\mathbb{P}.$

\item $F\in\left[  \omega_{2}\right]  ^{<\omega}.$

\item $s_{\alpha}\in T_{p}$ for every $\alpha\in F$ (where $p=\left(
T_{p},R_{p},\phi_{p}\right)  $).
\end{enumerate}

We define $\left(  p,\left\langle s_{\alpha}\right\rangle _{\alpha\in
F}\right)  \leq\left(  q,\left\langle t_{\alpha}\right\rangle _{\alpha\in
G}\right)  $ if the following conditions hold:

\begin{enumerate}
\item $p\leq_{\mathbb P} q.$

\item $G\subseteq F.$

\item $t_{\alpha}\subseteq s_{\alpha}$ for every $\alpha\in G.$
\end{enumerate}
\end{definition}

It is easy to see that $\mathbb{U}$ is forcing equivalent to $\mathbb{C}%
_{\omega_{2}}.$ Moreover, $\mathbb{U}$ is forcing equivalent to first forcing with
$\mathbb{P}$ and then adding $\omega_{2}$-Cohen reals. 
Given $\alpha<\omega_{2}$ we define $\dot{A}_{\alpha}$ to be
the set $\{n\mid\dot{\phi}_{gen}\left(  \dot{c}_{\alpha}\upharpoonright n\right)
=1\}$ (where $\dot{c}_{\alpha}$ is the name for the $\alpha$-th Cohen real). It
is easy to see that $\mathcal{\dot{A}}=\{\,\dot{A}_{\alpha}\mid\alpha
<\omega_{2}\}$ is forced to be an almost disjoint family of size $\omega_{2}.$

\begin{theorem}\label{cohen-main}
In the Cohen model, there is an almost disjoint family of size $\omega_{2}$ that
does not contain uncountable $\R$-embeddable subfamilies.
\end{theorem}

\begin{proof}
Since $\mathbb{U}$ is forcing equivalent to $\mathbb{C}_{\omega_{2}},$ we can
think of the Cohen model as the model obtained after forcing with $\mathbb{U}$
over a model of the Continuum Hypothesis.  Let $\mathcal{A}$ be the almost disjoint family that was  defined above. We
argue by contradiction, so assume that there is $\mathcal{\dot{B}=}\{\dot
{A}_{\dot{\alpha}_{\xi}}\mid\xi\in\omega_{1}\}$ and $\dot{f}$ such that
$\dot{f}$ is forced to be an embedding $\Psi(\mathcal{\dot{B})}$ into $2^{\omega}$
as in Lemma \ref{equivalences} (6). For every
$\xi\in\omega_{1},$ we may find $r_{\xi}=(p_{\xi},\langle s_{\eta}^{r\xi}\rangle_{\eta\in F_{\xi}})$$\in\mathbb{U}$ and
$\beta_{\xi}$ with the following properties:

\begin{enumerate}
\item $r_{\xi}\Vdash\dot{\alpha}_{\xi}=\beta_{\xi}.$

\item $\beta_{\xi}\in F_{\xi}.$

\item $s_{\beta_{\xi}}^{r_{\xi}}$ $^{\frown}0,$ $s_{\beta_{\xi}}^{r_{\xi}}$
$^{\frown}1\in T_{p_{\xi}}$ (where $p_{\xi}=(T_{p_{\xi}},R_{p_{\xi}},\phi_{p_{\xi}})$).
\end{enumerate}

By the $\Delta$-system lemma, (see \cite{kunen} Lemma III.2.6) we may find
$p\in\mathbb{P},$ $R\in\left[  \omega_{2}\right]  ^{<\omega},$ $W\in\left[
\omega_{2}\right]  ^{\omega_{1}}$ and $s\in\N^{<\omega}$ with the
following properties:

\begin{enumerate}
\item $p_{\xi}=p$ for every $\xi\in W.$

\item $\left\{  F_{\xi}\mid\xi\in W\right\}  $ forms a $\Delta$-system with root
$R.$

\item $s_{\eta}^{r_{\xi}}=s_{\eta}^{r_{\xi^{\prime}}}$
 for every
$\xi,\xi^{\prime}\in W$ and $\eta\in R.$

\item $s=s_{\beta_{\xi}}^{r_{\xi}}$ for every $\xi\in W.$

\end{enumerate}

It is easy to see that $\left\{  r_{\xi}\mid\xi\in W\right\}  \subseteq
\mathbb{U}$ is a centered set. Let $\left\{  H_{\alpha}\mid\alpha\in\omega
_{1}\right\}  \subseteq\left[  W\right]  ^{2}$ be a pairwise disjoint family.
For every $\alpha\in\omega_{1}$ we find $r_{\alpha}^{\prime}=(p_{\alpha}^{\prime},\langle u_{\eta}^{r_{\alpha}%
^{\prime}}\rangle_{\eta\in F_{\alpha}^{\prime}})$$\in\mathbb{U},$ $t_{\alpha}$ and
$z_{\alpha}$ with the following properties:

\begin{enumerate}
\item $r_{\alpha}^{\prime}\leq r_{\xi_{1}},r_{\xi_{2}}$ where $H_{\alpha
}=\left\{  \xi_{1},\xi_{2}\right\}  $ and $\xi_{1}<\xi_{2}.$

\item $s^{\frown}0\subseteq u_{\beta_{\xi_{1}}}^{r_{\alpha}^{\prime}}.$

\item $s^{\frown}1\subseteq u_{\beta_{\xi_{2}}}^{r_{\alpha}^{\prime}}.$

\item $t_{\alpha},$ $z_{\alpha}\in\N^{<\omega}$ are incompatible.

\item $r_{\alpha}^{\prime}\Vdash t_{\alpha}\subseteq\dot{f}(\dot{A}_{\beta_{\xi_{1}%
}})\wedge z_{\alpha}\subseteq\dot{f}(\dot{A}_{\beta_{\xi_{2}}}).$

\end{enumerate}

The last condition (5) can be obtained  since $\dot f$ is forced to be injective
when restricted to $\dot{\mathcal B}$ as in Lemma \ref{equivalences} (6).
Once again, we can find $W_{0}\in\left[  \omega_{2}\right]  ^{\omega_{1}},$
$p^{\prime}\in\mathbb{P},$ $s_{0},s_{1}\in\N^{<\omega},$ $R^{\prime\prime
}\in\left[  \omega_{2}\right]  ^{<\omega},t,z$ such that for every $\alpha\in
W_{0}$ the following holds:

\begin{enumerate}
\item $p_{\alpha}^{\prime}=p^{\prime}.$

\item $t_{\alpha}=t$ and $z_{\alpha}=z$.

\item $\left\{  F_{\alpha}^{\prime}\mid\alpha\in W_{0}\right\}  $ forms a
$\Delta$-system with root $R^{\prime}.$

\item $u_{\eta}^{r_{\alpha}^{\prime}}=u_{\eta}^{r_{\delta}^{\prime}}$
 for every $\alpha,\delta\in W_{0}$ and $\eta\in
R^{\prime}.$

\item $s_{0}=u_{\beta_{\xi_{1}}}^{r_{\alpha}^{\prime}}$ and $s_{1}=u_{\beta_{\xi
_{2}}}^{r_{\alpha}^{\prime}}$ 
 for every
$\alpha\in W_{0}$ where $H_{\alpha}=\left\{  \xi_{1},\xi_{2}\right\}  .$

\end{enumerate}

Once again, the set $\left\{  r_{\alpha}^{\prime}\mid\alpha\in W_{0}\right\}
\subseteq\mathbb{U}$ is centered. Let $\mathcal M$ be a countable elementary submodel
of some \textsf{H}$\left(  \kappa\right)  $ (where $\kappa$ is a sufficiently
big cardinal) containing all objects that have been defined so far. Let $\gamma\in
\mathcal M\cap W_{0}$ and $\delta\in W\setminus \mathcal M.$ Find $m\in\N$ such that
$s^{\frown}m\notin T_{p^{\prime}},$ let $\widehat{s}$ be a sequence extending $s^{\frown}m$
such that $\left\vert \widehat{s}\right\vert =\left\vert s_{0}\right\vert
=\left\vert s_{1}\right\vert$. Then we find $\widehat{r}=(\widehat{p},\langle y_{\eta}^{\widehat{r}}\rangle_{\eta\in F})$
 with the following properties:

\begin{enumerate}
\item $\widehat{r}\leq r_{\gamma}^{\prime},r_{\delta}.$

\item $T_{p^{\prime}}\cup\left\{  \widehat{s}\right\}  \subseteq
T_{\widehat{p}}$ (where $\widehat{p}=(T_{\widehat{p}},R_{\widehat{p}},F_{\widehat{p}})$).

\item $F=F_{\gamma}^{\prime}\cup F_{\delta}.$

\item $y_{\beta_{\delta}}^{\widehat{r}}=\widehat{s}.$

\item $\left\{  s_{0},\widehat{s}\right\}  ,$ $\left\{  s_{1},\widehat
{s}\right\}  \notin R_{p^{\prime}}.$
\end{enumerate}

We claim that $\widehat{r}$ forces that $\dot{f}[\dot{A}_{\beta_{\delta}}]$
has infinitely many elements below $t$ and infinitely many elements below $z,$
this will be a contradiction. Let $\widehat{r}_{1}\leq\widehat{r}$ and
$k\in\N,$ it will be enough to prove that we can extend $\widehat{r}_{1}$
to a condition that forces that there is $l>k$ such that $l$ is in $\dot
{A}_{\beta_{\delta}}$ and its image under $\dot{f}$ will be an extension of
$t$ whose height is bigger than $k$ (the case of $z$ is similar). Let
$\alpha\in \mathcal M\cap W_{0}$ such that $supp(\widehat{r_{1}})\cap M$ and $F_{\alpha}^{\prime}\setminus R^{\prime}$ are disjoint.
 Let $\widehat{r}_{2}$ be the greatest lower bound of
$\widehat{r}_{1}\cap \mathcal M$ and $r_{\alpha}^{\prime},$ note that $\widehat{r}%
_{2}\in \mathcal M.$ Let $e\in\omega$ such that $s_{0}$ $^{\frown}e$ has not been used
and let $v$ be extending $s_{0}$ $^{\frown}e$ such that $\left\vert v\right\vert =\left\vert s_{\beta_{\delta}}^{\widehat{r_{1}}%
}\right\vert $ 
and $\widehat{r}_{3}$ such that $s_{\beta_{\xi_{1}}}^{\widehat{r_{3}}}=v$
 (where $H_{\alpha}=\left\{  \xi_{1},\xi_{2}\right\}  $) and
$\{v,s_{\beta_{\delta}}^{\widehat{r_{1}}}\}\notin R_{\widehat{r_{3}}}.$ Since $\widehat{r}_{3},\dot{f}\in \mathcal M$ we can find
$\widehat{r}_{4}\in \mathcal M$ such that $\widehat{r}_{4}\leq\widehat{r}_{3}$ and
$l>k$ such that $\widehat{r_{4}}\Vdash l\in\dot{A}_{\beta_{\xi_{1}}}\wedge t\subseteq\dot
{f}\left(  l\right)  .$ Since the support of
$\widehat{r}_{4}$ is contained in $\mathcal M,$ then it is compatible with $\widehat
{r}_{1}.$ Since $\{v,s_{\beta_{\delta}}^{\widehat{r_{1}}}\}\notin R_{\widehat{r_{3}}},$ we can find a common extension that
forces that $l$ is in $\dot{A}_{\beta_{\delta}}.$
\end{proof}

\smallskip

The above family clearly does not have $\omega_1$-controlled $\R$-embedding
 property but a much stronger fact  concerning $\omega_1$-controlled $\R$-embedding 
property can be proved  in  the Cohen model.

\begin{theorem}\label{cohen-no-controlled} In the Cohen model,
no uncountable almost disjoint family $\A$   has $\omega_1$-controlled $\R$-embedding property.
\end{theorem}
\begin{proof} Let 
$\{c_\alpha\mid\alpha<\omega_2\}$ be the sequence of Cohen reals generating the Cohen model. 
Let $\F$ be an uncountable almost disjoint family.
For every $A\in \A$ there is a countable 
$X_A\subseteq\omega_2$ such that $A\in V[\{c_\alpha\mid \alpha\in X_A\}]$.
Define $\phi: \A\rightarrow 2^\omega$ by $\phi(A)=c_{\alpha_A}$ where $\alpha_A\not \in
X_A$ and all $\alpha_A$'s are distinct. 

Suppose that $f:\N\rightarrow 2^\omega$. There is a countable $Y\subseteq \omega_2$ such that
$f\in V[\{c_\alpha\mid \alpha\in Y\}]$. As $\A$ is uncountable, there is $A\in \A$ such
that $\alpha_A\not\in Y$, so $\alpha_A\not\in X_A\cup Y$. 
Hence $\lim_{n\in A}f(n)\not=c_{\alpha_A}=\phi(A)$, proving that 
$\A$ does not have $\omega_1$-controlled property.
\end{proof}

\begin{remark}
The above proofs remains valid for any  finite support product of not less than $2^\omega$
 c.c.c. forcings in place of the Cohen forcing.
\end{remark}

\section{$\R$-embeddability in the Sacks model}

By the \emph{Sacks model }we mean the model obtained by adding $\omega_{2}%
$-Sacks reals (with countable support) to a model of the \textsf{GCH}. Recall
that a tree $p\subseteq2^{<\omega}$ is a \emph{Sacks tree }if every node of
$p$ can be extended to a splitting node. We denote by $\mathbb{S}$ the
collection of all Sacks tree and we order it by inclusion. Given $\alpha
\leq\omega_{2}$ we denote by $\mathbb{S}_{\alpha}$ the countable support
iteration of $\mathbb{S}$ of length $\alpha.$\emph{ }We will now prove that
\ in the Sacks model, every almost disjoint family of size continuum contains an
$\R$-embeddable family of the same size. We will need to recall
some important notions and results on Sacks forcing. For more of this forcing
notion the reader may consult \cite{baumgartner-laver},
\cite{life} and \cite{miller}.

\begin{definition}
Let $\alpha\leq\omega_{2},$ $n,m\in\N.$

\begin{enumerate}
\item Given $p,q\in\mathbb{S}$ we say that $\left(  p,m\right)  \leq\left(
q,n\right)  $ if the following holds:

\begin{enumerate}
\item $p\leq q.$

\item $n\leq m.$

\item $q_{n}=p_{n}.$

\item If $n<m$ then for every $s\in q_{n}$ there are distinct $t_{0},t_{1}\in
p_{m}$ such that $s\subseteq t_{0},t_{1}.$
\end{enumerate}

\item Given $p,q\in\mathbb{S}_{\alpha}$ and $F\in\left[  \alpha\right]
^{<\omega}$ we say that $\left(  p,m\right)  \leq_{F}\left(  q,n\right)  $ if
the following holds:

\begin{enumerate}
\item $p\leq q.$

\item $n\leq m.$

\item if $\beta\in F$ then $p\upharpoonright\beta\Vdash\left(  p\left(
\beta\right)  ,m\right)  \leq\left(  q\left(  \beta\right)  ,n\right).$
\end{enumerate}
\end{enumerate}
\end{definition}

We will often use the following result:

\begin{lemma}
[Fusion lemma \cite{baumgartner-laver}]Let $\alpha\leq\omega_{2}$ and $\left\{
\left(  p_{i},F_{i},n_{i}\right)  \mid i\in\N\right\}  $ be a family such
that for every $i\in\N$ the following holds:

\begin{enumerate}
\item $p_{i}\in\mathbb{S}_{\alpha}.$

\item $F_{i}\in\left[  \alpha\right]  ^{<\omega}.$

\item $F_{i}\subseteq F_{i+1}.$

\item $n_{i}<n_{i+1}.$

\item $\left(  p_{i+1},n_{i+1}\right)  \leq_{F_{i}}\left(  p_{i},n_{i}\right)
.$

\item $%
{\textstyle\bigcup\limits_{j\in\N}}
F_{j}=%
{\textstyle\bigcup\limits_{j\in\N}}
supp\left(  p_{j}\right)  $
\end{enumerate}

Define $p$ such that $supp\left(  p\right)  =%
{\textstyle\bigcup\limits_{j\in\N}}
supp\left(  p_{j}\right)  $ and if $\beta\in supp\left(  p\right)  $ then
$p\left(  \beta\right)  $ is a $\mathbb{S}_{\beta}$-name for the intersection
of $\left\{  p_{i}\left(  \beta\right)  \mid\beta\in supp\left(  p_{i}\right)
\right\}  .$ Then $p\in\mathbb{S}_{\alpha}$ and $p\leq p_{i}$ for every
$i\in\N.$
\end{lemma}

If $p\in\mathbb{S}$ and $s\in2^{<\omega}$ we define $p_{s}=\left\{  t\in p\mid
t\subseteq s\vee s\subseteq t\right\}  .$ Note that $p_{s}$ is a Sacks tree if
and only if $s\in p.$

\begin{definition}
Let $p\in\mathbb{S}_{\alpha},$ $F\in\left[  supp\left(  p\right)  \right]
^{<\omega}$ and $\sigma:F\longrightarrow2^{n}.$ We define $p_{\sigma}$ as follows:

\begin{enumerate}
\item $supp\left(  p_{\sigma}\right)  =supp\left(  p\right)  .$

\item Letting $\beta<\alpha$ the following holds:

\begin{enumerate}
\item $p_{\sigma}\left(  \beta\right)  =p\left(  \beta\right)  $ if
$\beta\notin F.$

\item $p_{\sigma}\left(  \beta\right)  =p\left(  \beta\right)  _{\sigma\left(
\beta\right)  }$ if $\beta\in F.$
\end{enumerate}
\end{enumerate}
\end{definition}

Similar to previous situation, $p_{\sigma}$ is not necessarily a condition of
$\mathbb{S}_{\alpha}.$ We will say that $\sigma:F\longrightarrow2^{n}$ is
\emph{consistent with }$p$ if $p_{\sigma}\in\mathbb{S}_{\alpha}.$ A condition
$p$ is $\left(  F,n\right)  $\emph{-determined }if for every $\sigma
:F\longrightarrow2^{n}$ either $\sigma$ is consistent with $p$ or there is
$\beta\in F$ such that $\sigma\upharpoonright\left(  F\cap\beta\right)  $ is
consistent with $p$ and $(p\upharpoonright\beta)_{\sigma\upharpoonright\left(
F\cap\beta\right)  }\Vdash\sigma\left(  \beta\right)  \notin p\left(
\beta\right).$

We say that $p\in\mathbb{S}_{\alpha}$ is \emph{continous }if for every
$F\in\left[  supp\left(  p\right)  \right]  ^{<\omega}$ and for every
$n\in\N$ there are $G$ and $m$ such that the following holds:

\begin{enumerate}
\item $G\in\left[  supp\left(  p\right)  \right]  ^{<\omega}.$

\item $F\subseteq G.$

\item $n<m.$

\item $p$ is $\left(  G,m\right)  $-determined.
\end{enumerate}

We will need the following lemmas:\ \ 

\begin{lemma}
[\cite{baumgartner-laver}]Let $p\in\mathbb{S}_{\alpha},n\in\N$ and
$F\in\left[  supp\left(  p\right)  \right]  ^{<\omega}.$ There is $\left(
q,m\right)  \leq_{F}\left(  p,n\right)  $ such that $q$ is $\left(
F,n\right)  $-determined.
\end{lemma}

\begin{lemma}
[\cite{life}]For every $p\in\mathbb{S}_{\alpha}$ there is a
continous $q\leq p$ such that $q$ is continous.
\end{lemma}

\bigskip Let $p$ be a continuous condition. We say that $\left\{  \left(
F_{i},n_{i},\Sigma_{i}\right)  \mid i\in\omega\right\}  $ is \emph{a
representation of }$p$ if the following holds:

\begin{enumerate}
\item $F_{i}\in\left[  supp\left(  p\right)  \right]  ^{<\omega},$ $n_{i}%
\in\omega.$

\item $F_{i}\subseteq F_{i+1}$ and $n_{i}<n_{i+1}.$

\item $supp\left(  p\right)  =%
{\textstyle\bigcup\limits_{i\in\N}}
F_{i}.$

\item $p$ is $\left(  F_{i},n_{i}\right)  $-determined for every $i\in\omega.$

\item $\Sigma_{i}$ is the set of all $\sigma:F_{i}\longrightarrow2^{n_{i}}$
such that $\sigma$ is consistent with $p.$
\end{enumerate}

Note that if $\left\{  \left(  F_{i},n_{i},\Sigma_{i}\right)  \mid i\in
\omega\right\}  $ is a representation of $p$ and $f:\omega\longrightarrow
\omega$ is an increasing function, then $\left\{  \left(  F_{f\left(
i\right)  },n_{f\left(  i\right)  },\Sigma_{f\left(  i\right)  }\right)  \mid
i\in\N\right\}  $ is also a representation of $p.$ It is also easy to see
that if $p$ is continuous with representation $\left\{  \left(  F_{i}%
,n_{i},\Sigma_{i}\right)  \mid i\in\N\right\}  $ and $\sigma\in\Sigma
_{i},$ then $p_{\sigma}$ is also a continuous condition. Given a continuous
condition $p\in\mathbb{S}_{\alpha}$ and $R=\left\{  \left(  F_{i},n_{i}%
,\Sigma_{i}\right)  \mid i\in\N\right\}  $ a representation of $p,$ we
define $\left[  p\right]  _{R}$ as the set of all $\left\langle y_{\beta
}\right\rangle _{\beta\in supp\left(  p\right)  }\in\left(  2^{\omega}\right)
^{supp\left(  p\right)  }$ such that for every $i\in\omega$ the function
$\sigma:F_{i}\longrightarrow2^{n_{i}}$ given by $\sigma\left(  \beta\right)
=y_{\beta}\upharpoonright n_{i}$ belongs to $\Sigma_{i}.$

\begin{lemma}
Let $p\in\mathbb{S}_{\alpha}$ be a continuous condition. If $R=\left\{  \left(
F_{i},n_{i},\Sigma_{i}\right)  \mid i\in\N\right\}  $ and $R^{\prime
}=\left\{  \left(  G_{i},m_{i},\Pi_{i}\right)  \mid i\in\N\right\}  $ are
two representations of $p,$ then $\left[  p\right]  _{R}=\left[  p\right]
_{R^{\prime}}.$
\end{lemma}

\begin{proof}
We argue by contradiction, assume that there is $\overline{y}=\left\langle
y_{\beta}\right\rangle _{\beta<a}\in\left[  p\right]  _{R}\setminus\left[
p\right]  _{R^{\prime}}.$ Since $\overline{y}\notin\left[  p\right]
_{R^{\prime}}$ there must be $i\in\omega$ such that the function $\sigma
:G_{i}\longrightarrow2^{m_{i}}$ given by $\sigma\left(  \beta\right)
=y_{\beta}\upharpoonright m_{i}$ is not in $\Pi_{i}$, i.e. $\sigma$ is not
consistent with $p.$ Since $p$ is $\left(  G_{i},m_{i}\right)  $-determined,
there is $\beta\in G_{i}$ such that $\sigma\upharpoonright\left(  G_{i}%
\cap\beta\right)  $ is consistent with $p$ but $p_{\sigma\upharpoonright
\left(  G_{i}\cap\beta\right)  }\Vdash\sigma\left(  \beta\right)  \notin
p\left(  \beta\right).$ Let $j\in\omega$ such that
$G_{i}\subseteq F_{j}$ and $m_{i}<n_{j}.$ Since $\overline{y}\in\left[
p\right]  _{R}$ we know that the function $\tau:F_{j}\longrightarrow2^{n_{j}}$
given by $\tau\left(  \xi\right)  =y_{\xi}\upharpoonright n_{j}$ is consistent
with $p.$ It is clear that $p_{\tau\upharpoonright\left(  F_{j}\cap
\beta\right)  }\leq p_{\sigma\upharpoonright\left(  G_{i}\cap\beta\right)  }$
and $\sigma\left(  \beta\right)  \subseteq\tau\left(  \beta\right)  $ so
$p_{\tau\upharpoonright\left(  F_{j}\cap\beta\right)  }$ forces that
$\tau\left(  \beta\right)  $ is not in $p\left(  \beta\right)  ,$ which
contradicts the fact that $\tau$ is consistent with $p.$
\end{proof}

In light of the previous result, we will omit the subscript and only write
$\left[  p\right]  $ to refer to $\left[  p\right]  _{R}$ where $R$ is any
representation of $p.$ It is easy to see that if $p\in\mathbb{S}_{\alpha}$ is
a continuous condition then $\left[  p\right]  $ is a compact set and
$p\Vdash\overline{s}_{gen}\upharpoonright supp\left(  p\right)  \in\left[
p\right]$ (where $\overline{s}_{gen}$ is the sequence of
generic reals). Let $S\in\left[  \omega_{2}\right]  ^{\omega}$ and
$\sigma:F\longrightarrow2^{<\omega}$ where $F\in\left[  S\right]  ^{<\omega},$
we define $\left\langle \sigma\right\rangle _{S}$ as the set  $\{\left\langle
y_{\beta}\right\rangle _{\beta\in S}\in\left(  2^{\omega}\right)  ^{S}%
\mid\forall\beta\in F\left(  \sigma\left(  \beta\right)  \subseteq y_{\beta
}\right)  \}.$ Note that this is family of sets are the basis for the topology
of $\left(  2^{\omega}\right)  ^{S}.$ The following result is well known:

\begin{lemma}
[Continuous reading of names for Sacks forcing]Let $\alpha<\omega_{2},$
$p\in\mathbb{S}_{\alpha}$ and $\dot{x}$ be a $\mathbb{S}_{\alpha}$-name such
that $p\Vdash\dot{x}\in\left[  \omega\right]  ^{\omega}.$
There is a continuous condition $q\leq p$ and a continuous function $F:\left[
q\right]  \longrightarrow\left[  \N\right]  ^{\omega}$ such that
$q\Vdash$$F\left(  \dot{s}_{gen}\upharpoonright supp\left(  q\right)  \right)
=\dot{x}$ (where $\dot{s}_{gen}$ is the name for the generic real). 
\end{lemma}

We will need the following notion:

\begin{definition}
Let $\mathcal{C},\mathcal{D}$ be two subfamilies of $\wp\left(  \N\right)
.$ We say that the pair $\left(  \mathcal{C},\mathcal{D}\right)  $ \emph{is
decisive }if one of the following two conditions hold:

\begin{enumerate}
\item Either $c\cap d$ is infinite for every $c\in\mathcal{C}$ and
$d\in\mathcal{D}$ or

\item $c\cap d$ is finite for every $c\in\mathcal{C}$ and $d\in\mathcal{D}.$
\end{enumerate}
\end{definition}

Note that if the second alternative holds and $\mathcal{C}$ and $\mathcal{D}$
are both compact, then there is an $m$ such that $c\cap d\subseteq m$ for
every $c\in\mathcal{C}$ and $d\in\mathcal{D}.$

\begin{lemma}
Let $p,q$ be two continous conditions in $\mathbb{S}_{\alpha}$ such that
$supp\left(  p\right)  =supp\left(  q\right)  $ and $F:\left(  2^{\omega
}\right)  ^{supp\left(  p\right)  }\longrightarrow\left[  \N\right]
^{\omega}$ a continuous function. There are $p^{\prime},q^{\prime}\in
\mathbb{S}_{\alpha}$ such that the following holds:

\begin{enumerate}
\item $p^{\prime}\leq p$ and $q^{\prime}\leq q.$

\item $supp\left(  p\right)  =supp\left(  q\right)  =supp\left(  p^{\prime
}\right)  =supp\left(  q^{\prime}\right)  .$

\item The pair $\left(  F\left[  \left[  p^{\prime}\right]  \right]  ,F\left[
\left[  q^{\prime}\right]  \right]  \right)  $ is decisive.
\end{enumerate}
\end{lemma}

\begin{proof}
We proceed by cases, the first case is that there are $p^{\prime}\leq p,$ $q^{\prime}\leq q$ with $supp\left(  p\right)
=supp\left(  q\right)  =supp\left(  p^{\prime}\right)  =supp\left(  q^{\prime
}\right)  $ and $m\in\N$ such $F\left(  \overline{y}\right)  \cap F\left(
\overline{z}\right)  \subseteq m$ for every $\overline{y}\in\left[  p^{\prime
}\right]  $ and $\overline{z}\in\left[  q^{\prime}\right]  .$

In this case it is clear that the pair $\left(  F\left[  \left[  p^{\prime
}\right]  \right]  ,F\left[  \left[  q^{\prime}\right]  \right]  \right)  $ is decisive. The second case is that for every $p^{\prime}\leq p,$ $q^{\prime}\leq q$ with $supp\left(  p\right)
=supp\left(  q\right)  =supp\left(  p^{\prime}\right)  =supp\left(  q^{\prime
}\right)  $ and $m\in\N$ there are $\overline{y}\in\left[  p^{\prime
}\right]  ,$ $\overline{z}\in\left[  q^{\prime}\right]  $ and $k>m$ such that
$k\in F\left(  \overline{y}\right)  \cap F\left(  \overline{z}\right)  .$

Let $supp\left(  p\right)  =\left\{  \alpha_{n}\mid n\in\N\right\}  .$ We
will now recursively build the two sequences $\left\{  \left(  p^{n},m_{n}%
,F_{n}\right)  \mid n\in\N\right\}  $ and $\left\{  \left(  q^{n}%
,k_{n},G_{n}\right)  \mid n\in\N\right\}  $ such that for every
$n\in\N$ the following holds:

\begin{enumerate}
\item $p^{0}=p$ and $q^{0}=q,$ $F_{0}=G_{0}=\emptyset.$

\item $F_{n}\in\left[  supp\left(  p\right)  \right]  ^{<\omega},$
$F_{n}\subseteq F_{n+1}$ and $\alpha_{n}\in F_{n+1}.$

\item $G_{n}\in\left[  supp\left(  q\right)  \right]  ^{<\omega},$
$G_{n}\subseteq G_{n+1}$ and $\alpha_{n}\in G_{n+1}.$

\item $m_{0}=k_{0}=0.$

\item $m_{n}<m_{n+1}$ and $k_{n}<k_{n+1}.$

\item $p^{n}$ and $q^{n}$ are continous conditions.

\item $supp\left(  p^{n}\right)  =supp\left(  q^{n}\right)  =supp\left(
p\right)  .$

\item $\left(  p^{n+1},m_{n+1}\right)  \leq_{F_{n+1}}\left(  p^{n}%
,m_{n}\right)  $ and $\left(  q^{n+1},k_{n+1}\right)  \leq_{G_{n+1}}\left(
q^{n},k_{n}\right)  .$

\item $p^{n+1}$ is $\left(  F_{n+1},m_{n+1}\right)  $-determined and $q^{n+1}$
is $\left(  G_{n+1},k_{n+1}\right)  $-determined.

\item For every $\sigma:F_{n}\longrightarrow2^{m_{n}}$ and $\tau
:G_{n}\longrightarrow2^{k_{n}}$ if $\sigma$ is consistent with $p^{n}$ and
$\tau$ is consistent with $q^{n}$ then there is $l>n$ such that $l\in F\left(
\overline{y}\right)  \cap F\left(  \overline{z}\right)  $ for every
$\overline{y}\in\left[  p_{\sigma}^{n}\right]  $ and $\overline{z}\in\left[
q_{\tau}^{n}\right]  .$
\end{enumerate}

Assume we are at step $n+1.$ Since both $p^{n}$ and $q^{n}$ are continuous
conditions, we can find $F_{n+1},G_{n+1},m_{n+1}$ and $k_{n+1}$ with the
following properties:

\begin{enumerate}
\item $F_{n}\cup\left\{  \alpha_{n}\right\}  \subseteq F_{n+1}$ and $G_{n}%
\cup\left\{  \alpha_{n}\right\}  \subseteq G_{n+1}.$

\item $m_{n}<m_{n+1},$ $k_{n}<k_{n+1}.$

\item $p^{n}$ is $\left(  F_{n+1},m_{n+1}\right)  $-determined and $q^{n}$ is
$\left(  G_{n+1},k_{n+1}\right)  $-determined.
\end{enumerate}

Let $W=\left\{  \left(  \sigma_{i},\tau_{i}\right)  \right\}  _{i<u}$
enumerate all pairs $\left(  \sigma,\tau\right)  $ for which $\sigma
:F_{n+1}\longrightarrow2^{m_{n+1}}$ and $\tau:G_{n+1}\longrightarrow
2^{k_{n+1}}.$ We recursively find a sequence $\left\{  \left(  p_{1}^{i}%
,q_{1}^{i}\right)  \mid i<u+1\right\}  $ such that for every $i<u$ the
following holds:

\begin{enumerate}
\item $p^{n}=p_{1}^{0}$ and $q^{n}=q_{1}^{0}.$

\item $\left(  p_{1}^{i+1},m_{n+1}\right)  \leq_{F_{n+1}}\left(  p_{1}%
^{i},m_{n+1}\right)  $ and $\left(  q_{1}^{i+1},k_{n+1}\right)  \leq_{G_{n+1}%
}\left(  q_{1}^{i},k_{n+1}\right)  .$

\item $p_{1}^{i}$ and $q_{1}^{i}$ are continous.

\item $supp\left(  p_{1}^{i}\right)  =supp\left(  q_{1}^{i}\right)  =supp\left(
p\right)  .$

\item If $\sigma_{i}$ is consistent with $p_{1}^{i+1}$ and $\tau_{i}$ is
consistent with $q_{1}^{i+1}$ then there is $l>n$ such that $l\in F\left(
\overline{y}\right)  \cap F\left(  \overline{z}\right)  $ for every
$\overline{y}\in\lbrack\left(  p_{1}^{i+1}\right)  _{\sigma_{i}}]$ and
$\overline{z}\in\lbrack\left(  q_{1}^{i+1}\right)  _{\tau_{i}}].$
\end{enumerate}

Assume we are at step $i.$ In case either $\sigma_{i}$ is not consistent with
$p_{1}^{i}$ or $\tau_{i}$ is not consistent with $q_{1}^{i}$ we simply define
$p_{1}^{i+1}=p_{1}^{i}$ and $q_{1}^{i+1}=q_{1}^{i}.$ Assume $\sigma_{i}$ is
consistent with $p_{1}^{i}$ and $\tau_{i}$ is consistent with $q_{1}^{i}.$ By
the hypothesis, there are $l>n,$ $\overline{y}\in\lbrack\left(  p_{1}%
^{i}\right)  _{\sigma_{i}}]$ and $\overline{z}\in\lbrack\left(  q_{1}%
^{i}\right)  _{\tau_{i}}]$ such that $k\in F\left(  \overline{y}\right)  \cap
F\left(  \overline{z}\right)  .$ Since $F$ is a continous function, we can
find $p_{1}^{i+1}$ and $q_{1}^{i+1}$ with the following properties:

\begin{enumerate}
\item $\sigma_{i}$ is consistent with $p_{1}^{i+1}.$

\item $\tau_{i}$ is consistent with $q_{1}^{i+1}.$

\item For every $\overline{y_{1}}\in\lbrack\left(  p_{1}^{i+1}\right)
_{\sigma_{i}}]$ and $\overline{z_{1}}\in\lbrack\left(  q_{1}^{i+1}\right)
_{\tau_{i}}]$ it is the case that $k\in F\left(  \overline{y}\right)  \cap
F\left(  \overline{z}\right)  .$

\item $p_{1}^{i+1}$ and $q_{1}^{i+1}$ are continous.

\item $\left(  p_{1}^{i+1},m_{n+1}\right)  \leq_{F_{n+1}}\left(  p_{1}%
^{i},m_{n+1}\right)  $ and $\left(  q_{1}^{i+1},k_{n+1}\right)  \leq_{F_{n}%
}\left(  q_{1}^{i},k_{n+1}\right)  .$

\item $supp\left(  p_{1}^{i}\right)  =supp\left(  q_{1}^{i}\right)  =supp\left(
p\right)  .$
\end{enumerate}

We then define $p^{n+1}=p_{1}^{u+1}$ and $q^{n+1}=q_{1}^{u+1}.$

Let $p^{\prime}$ and $q^{\prime}$ be the respective fusion sequences. It is
easy to see that $F\left[  \overline{c}\right]  \cap F\left[  \overline
{e}\right]  $ is infinite for every $\overline{c}\in\left[  p^{\prime}\right]
$ and $\overline{e}\in\left[  q^{\prime}\right]  .$
\end{proof}

Note that if $p$ is continous and $\beta=min\left(  supp\left(  p\right)
\right)  $ then we may assume that $p\left(  \beta\right)  $ is a real Sacks
tree (not only a name).

\begin{proposition}
Let $p\in\mathbb{S}_{\alpha}$ be a continuous condition, $F:\left[  p\right]
\longrightarrow\left[  \N\right]  ^{\omega}$ a continuous function and
$\beta=min\left\{  supp\left(  \alpha\right)  \right\}  .$ Then there are $q$
$\in\mathbb{S}_{\alpha}$ with representation $\left\{  \left(  F_{i}%
,m_{i},\Sigma_{i}\right)  \mid i\in\N\right\}  $ such that the following holds:

\begin{enumerate}
\item $q\leq p.$

\item $supp\left(  q\right)  =supp\left(  p\right)  .$

\item $F_{0}=\left\{  \beta\right\}  .$

\item For every $i\in\N$ the following holds: for every $\sigma,\tau
\in\Sigma_{i}$ such that $\sigma\left(  \beta\right)  \neq\tau\left(
\beta\right)  ,$ the pair $\left(  F\left[  \left[  q_{\sigma}\right]
\right]  ,F\left[  \left[  q_{\tau}\right]  \right]  \right)  $ is decisive.
\end{enumerate}
\end{proposition}

\begin{proof}
Let $supp\left(  p\right)  =\left\{  \alpha_{n}\mid n\in\N\right\}  $ with
$\alpha_{0}=\beta$ $.$ We recursively build a sequence $\{(p^{n},m_{n}%
,F_{n})\mid n\in\N\}$ with the following properties:

\begin{enumerate}
\item $p^{0}=p.$

\item $F_{0}=\left\{  \beta\right\}  $ and $m_{0}=0.$

\item Each $p^{n}$ is continuous and $supp\left(  p^{n}\right)  =supp\left(
p\right)  .$

\item $F_{n}\in\left[  supp\left(  p\right)  \right]  ^{<\omega}$ and
$\alpha_{n}\in F_{n}.$

\item $\left(  p^{n+1},m_{n+1}\right)  $ $\leq_{F_{n}}\left(  p^{n}%
,m_{n}\right)  .$

\item $m_{n}<m_{n+1}.$

\item For every $\sigma,\tau:F_{n}\longrightarrow2^{m_{n}}$ such that
$\sigma\left(  \beta\right)  \neq\tau\left(  \beta\right)  $ and both are
consistent with $p^{n}$, the pair $\left(  F\left[  \left[  p_{\sigma}%
^{n}\right]  \right]  ,F\left[  \left[  p_{\tau}^{n}\right]  \right]  \right)
$ is decisive.
\end{enumerate}

Assume we are at step $n.$ We first find $F_{n+1}$ and $m_{n+1}>m_{n}$ such
that $F_{n}\cup\left\{  \alpha_{n}\right\}  \subseteq F_{n+1}$ and $p^{n}$ is
$\left(  F_{n+1},m_{n+1}\right)  $-determined. Let $W$ be the set of all pairs
$\left(  \sigma,\tau\right)  $ such that $\sigma,\tau:F_{n+1}\longrightarrow
2^{m_{n+1}},$ $\sigma\left(  \beta\right)  \neq\tau\left(  \beta\right)  $ and
both are consistent with $p^{n}.$ Enumerate $W=\left\{  \left(  \sigma
_{i},\tau_{i}\right)  \mid i\leq l\right\}  .$ We recursively build $\left\{
q_{i}\mid i\leq l\right\}  $ with the following properties:

\begin{enumerate}
\item Each $q_{i}$ is $\left(  F_{n+1},m_{n+1}\right)  $-determined and continuous.

\item $supp\left(  q_{i}\right)  =supp\left(  p\right)  .$

\item $\left(  q_{0},m_{n}\right)  \leq_{F_{n+1}}\left(  p^{n},m_{n+1}\right)
.$

\item $\left(  q_{i+1},m_{n+1}\right)  \leq_{F_{n}}\left(  q_{i}%
,m_{n+1}\right)  $ for $i<l.$

\item For each $i\leq l$ one of the following conditions hold:

\begin{enumerate}
\item Either $\sigma_{i}$ or $\tau_{i}$ is not consistent with $q_{i}$ or

\item the pair $(F[[\left(  q_{i}\right)  _{\sigma_{i}}]],F[[\left(
q_{i}\right)  _{\tau_{i}}]])$ is decisive.
\end{enumerate}
\end{enumerate}

Assume we are at step $i<l.$ In case that $\sigma_{i+1}$ or $\tau_{i+1}$ is
not consistent with $q_{i}$ we simply define $q_{i+1}=q_{i}.$ We now assume
both $\sigma_{i+1}$ and $\tau_{i+1}$ are consistent with $q_{i}.$ By applying
the previous lemma to $\left(  q_{i}\right)  _{\sigma_{i}}$ and $\left(
q_{i}\right)  _{\tau_{i}}$ we obtain $r_{0},$ $r_{1}$ continous conditions
with the following properties:

\begin{enumerate}
\item $r_{0}\leq\left(  q_{i}\right)  _{\sigma_{i}}.$

\item $r_{1}\leq\left(  q_{i}\right)  _{\tau_{i}}.$

\item $supp\left(  r_{0}\right)  =supp\left(  r_{1}\right)  =supp\left(
p\right)  .$

\item the pair $(F[[r_{0}]],F[[r_{1}]])$ is decisive.
\end{enumerate}

We now define the $r$ to be a Sacks tree with the following properties:

\begin{enumerate}
\item $r_{\sigma_{i+1}\left(  0\right)  }=r_{0}\left(  \beta\right)  .$

\item $r_{\tau_{i+1}\left(  0\right)  }=r_{1}\left(  \beta\right)  .$

\item $r_{s}=q_{i}\left(  \beta\right)  _{s}$ for $s\in q_{i}\left(
\beta\right)  _{m_{n}}$ and $s\notin\left\{  \sigma_{i+1}\left(  \beta\right)
,\tau_{i+1}\left(  \beta\right)  \right\}  .$
\end{enumerate}

Let $\dot{u}$ be a $\mathbb{S}$-name with the following properties:

\begin{enumerate}
\item $r_{0}\upharpoonright\left(  \beta+1\right)  \Vdash\dot{u}%
=\left\langle r_{0}\left(  \xi\right)  \right\rangle _{\xi>\beta
}.$

\item $r_{1}\upharpoonright\left(  \beta+1\right)  \Vdash\dot{u}%
=\left\langle r_{1}\left(  \xi\right)  \right\rangle _{\xi>\beta
}.$

\item $r^{\prime}\Vdash``\dot{u}=\left\langle q_{i}\left(  \xi\right)
\right\rangle _{\xi>\beta}\textquotedblright$ for every $r^{\prime}\leq
r\upharpoonright\left(  \beta+1\right)  $ that is incompatible with both
$r_{0}\upharpoonright\left(  \beta+1\right)  $ and $r_{1}\upharpoonright
\left(  \beta+1\right)  .$
\end{enumerate}

Let $q_{i+1}=$ $r^{\frown}\dot{u}.$ It is easy to see that $q_{i+1}$ has the
desired properties. Finally, we define $p^{n+1}=q_{l}.$ The fusion has the
desired properties.
\end{proof}

Let $a$ be a countable subset of $\omega_{2}$. We can define $\mathbb{S}_{a}$
as a countable support iteration of Sacks forcing with domain $a.$ Clearly,
$\mathbb{S}_{a}$ is isomorphic to $\mathbb{S}_{\delta}$ where $\delta$ is the
order type of $a.$ Note that if $p\in\mathbb{S}_{\omega_{2}}$ is a continuous
condition, then it can be seen as a condition of $\mathbb{S}_{supp\left(
p\right)  }.$ With this remark, it is easy to prove the following:

\begin{proposition}
Let $p\in\mathbb{S}_{\alpha}$ be a continuous condition that has a
representation $\left\{  \left(  F_{i},n_{i},\Sigma_{i}\right)  \mid
i\in\N\right\}  $ and $F:\left[  p\right]  \longrightarrow\left[
\N\right]  ^{\omega}$ a continuous function. Let $\alpha^{\ast}$ be the
order type of $supp\left(  p\right)  $ and $\pi:supp\left(  p\right)
\longrightarrow\alpha^{\ast}$ be the (unique) order isomorphism. There are
$q\in\mathbb{S}_{\alpha^{\ast}}$ and$\ $a continuous function $H:\left[
q\right]  \longrightarrow\left[  \N\right]  ^{\omega}$ with the following properties:

\begin{enumerate}
\item $supp\left(  q\right)  =\alpha^{\ast}.$

\item The set $\left\{  \left(  \pi\left[  F_{i}\right]  ,n_{i},\pi\Sigma_{i}\right)
\mid i\in\N\right\}  $ is a representation of $q$ (where $\pi\Sigma
_{i}=\left\{  \pi\sigma\mid\sigma\in\Sigma_{i}\right\}  $).

\item If $\overline{\pi}:\left(  2^{\omega}\right)  ^{supp\left(  p\right)
}\longrightarrow\left(  2^{\omega}\right)  ^{\alpha^{\ast}}$ denotes the
natural homeomorphism induced by $\pi.$ then $\overline{\pi}\upharpoonright
\left[  p\right]  $ is an homeomorphism and $F=H\overline{\pi}.$
\end{enumerate}
\end{proposition}

We will say that $\left(  p,F\right)  $ and $\left(  q,H\right)  $ are
\emph{isomorphic }if the previous conditions hold.

\begin{theorem}\label{sacks-main}
In the Sacks model, every almost disjoint family of size $\omega_{2}$ contains an
$\R$-embedabble subfamily of size $\omega_{2}.$
\end{theorem}

\begin{proof}
Let $\mathcal{\dot{A}=}\left\{  \,\dot{A}_{\alpha}\mid\alpha\in\omega
_{2}\right\}  $ be a $\mathbb{S}_{\omega_{2}}$-name for an almost disjoint family.
For every $\alpha<\omega_{2}$ we choose a pair $\left(  p_{\alpha},F_{\alpha
}\right)  $ with the following properties:

\begin{enumerate}
\item $p_{\alpha}$ is a continuous condition.

\item $F_{\alpha}:\left[  p_{\alpha}\right]  \longrightarrow\left[
\N\right]  ^{\omega}$ is a continuous function.

\item $p_{\alpha}\Vdash$$F_{\alpha}\left(  \overline{r}_{gen}\upharpoonright
supp\left(  p_{\alpha}\right)  \right)  =\dot{A}_{\alpha}$
\end{enumerate}

By the $\Delta$-system lemma, we can assume that $\left\{  supp\left(  p_{\alpha
}\right)  \mid\alpha\in\omega_{2}\right\}  $ forms a delta system with root
$R\in\left[  \omega_{2}\right]  ^{\omega}.$ Let $\delta\in\omega_{2}$ such
that $R\subseteq\delta.$ By a pruning argument, we may assume that
$R=supp\left(  p_{\alpha}\right)  \cap\delta$ for every $\alpha<\omega_{2}.$
Since $\mathbb{S}_{\omega_{2}}$ has the $\omega_{2}$-chain condition, there is
$p\in\mathbb{S}_{\omega_{2}}$ such that $p$ forces that the set $\{\alpha\mid
p_{\alpha}\in\dot{G}\}$ will have size $\omega_{2}$ (where $\dot{G}$ is the
name of the generic filter). Note that we may assume that $p\in\mathbb{S}%
_{\delta}$ (by increasing $\delta$ if needed).

Let $G_{0}\subseteq\mathbb{S}_{\delta}$ be a generic filter such that $p\in
G_{0}.$ We will now work in $V\left[  G_{0}\right]  .$ Let $W=\left\{
\alpha\mid\left(  p_{\alpha}\upharpoonright\delta\right)  \in G_{0}\right\}  $
which has size $\omega_{2}$ by the nature of $p.$ For every $\alpha\in W,$ let
$p_{\alpha}^{\prime}$ be the $\mathbb{S}_{\delta}$-name such that $p_{\alpha
}=\left(  p_{\alpha}\upharpoonright\delta\right)  ^{\frown}p_{\alpha}^{\prime
}.$ Note that we may view each $p_{\alpha}^{\prime}\left[  G_{0}\right]  $ as
a condition of $\mathbb{S}_{\omega_{2}}$ where $supp\left(  p_{\alpha}^{\prime
}\left[  G_{0}\right]  \right)  =supp\left(  p_{\alpha}\right)  \setminus
\delta.$ Let $\overline{r}=\left\langle r_{\beta}\right\rangle _{\beta<\delta
}$ be the generic sequence of reals added by $G_{0}.$ We can now define
$H_{\alpha}:\left[  p_{\alpha}^{\prime}\left[  G_{0}\right]  \right]
\longrightarrow\left[  \N\right]  ^{\omega}$ given by $H_{\alpha}\left(
\left\langle y_{\beta}\right\rangle \right)  =F_{\alpha}\left(  \left(
\overline{r}\upharpoonright supp\left(  p_{\alpha}\right)  \right)  ^{\frown
}\left\langle y_{\beta}\right\rangle \right)  $ which is a continous function.
By a previous lemma, for each $\alpha\in W$ we can find a continous condition
$q_{\alpha}$ and $\left\{  \left(  F_{i}^{\alpha},m_{i}^{\alpha},\Sigma
_{i}^{\alpha}\right)  \mid i\in\omega\right\}  $ a representation of
$q_{\alpha}$ with the following properties:

\begin{enumerate}
\item $q_{\alpha}\leq p_{\alpha}^{\prime}\left[  G_{0}\right]  .$

\item $supp\left(  q_{\alpha}\right)  =supp\left(  p_{\alpha}^{\prime}\left[
G_{0}\right]  \right)  .$

\item $F_{0}^{\alpha}=\left\{  \beta_{\alpha}\right\}  $ where $\beta_{\alpha
}=min\left(  supp\left(  p_{\alpha}^{\prime}\left[  G_{0}\right]  \right)
\right)  .$

\item For every $i\in\omega$ the following holds: for every $\sigma,\tau
\in\Sigma_{i}^{\alpha}$ such that $\sigma\left(  \beta_{\alpha}\right)
\neq\tau\left(  \beta_{{}}\right)  ,$ the pair $\left(  H_{\alpha}\left[
\left[  \left(  q_{\alpha}\right)  _{\sigma}\right]  \right]  ,H_{\alpha
}\left[  \left[  \left(  q_{\alpha}\right)  _{\tau}\right]  \right]  \right)
$ is decisive.
\end{enumerate}

Let $\alpha^{\ast}$ be the order type of $supp\left(  q_{\alpha}\right)  .$ For
each $\alpha\in W$ we find $q_{\alpha}^{\ast}\in\mathbb{S}_{\alpha^{\ast}}$
and $H_{\alpha}^{\ast}:\left[  q_{\alpha}^{\ast}\right]  \longrightarrow
\left[  \N\right]  ^{\omega}$ such that $\left(  q_{\alpha},H_{\alpha
}\right)  $ and $\left(  q_{\alpha}^{\ast},H_{\alpha}^{\ast}\right)  $ are
isomorphic. We can then find find $\gamma,$ $q^{\ast}\in\mathbb{S}_{\gamma}$
with representation $\left\{  \left(  F_{i},m_{i},\Sigma_{i}\right)  \mid
i\in\N\right\}  $ and a continous function $H:\left[  \gamma\right]
\longrightarrow\left[  \N\right]  ^{\omega}$ such that the set $W^{\prime
}\subseteq W$ consisting of all $\alpha$ such that $\alpha^{\ast}=\gamma,$
$q_{\alpha}^{\ast}=q^{\ast}$ and $H_{\alpha}^{\ast}=H$ has size $\omega_{2}.$

We first note that for every $i\in\N$ the following holds: for every
$\sigma,\tau\in\Sigma_{i}$ such that $\sigma\left(  0\right)  \neq\tau\left(
0\right)  ,$ the pair $\left(  H\left[  \left[  q_{\sigma}^{\ast}\right]
\right]  ,H\left[  \left[  q_{\tau}^{\ast}\right]  \right]  \right)  $ is
decisive, furthermore, $H\left(  \overline{y}\right)  \cap H\left(
\overline{z}\right)  $ is finite for every $\overline{y}\in q_{\sigma}^{\ast}$
and $\overline{z}\in q_{\tau}^{\ast}.$ It is decisive since $\left(
q_{\alpha},H_{\alpha}\right)  $ and $\left(  q^{\ast},H\right)  $ are
isomorphic, the second part of the claim follows since any pair of conditions
indexed by elements of $W^{\prime}$ have disjoint supports (and $\mathcal{A}$
is forced to be an almost disjoint family).

Given $s\in q^{\ast}\cap2^{n_{i}}$ let $B_{s}=%
{\textstyle\bigcup}
\left\{  \cup H\left[  \left[  q_{\sigma}\right]  \right]  \mid\sigma\in
\Sigma_{i}\wedge\sigma\left(  0\right)  =s\right\}  .$ Note that if $s$ and
$t$ are two different elements of $q^{\ast}\cap2^{n_{i}}$ then $B_{s}$ and
$B_{t}$ are almost disjoint. Let $T=\left\{  B_{s}\mid s\in q^{\ast}%
\cap2^{n_{i}}\wedge i\in\N\right\}  .$

Note that if $\alpha\in W^{\prime}$ then $q_{\alpha}\Vdash\dot{A}_{\alpha
}\subseteq%
{\textstyle\bigcap\limits_{s\subseteq\dot{r}_{\beta_{\alpha}}}}
B_{s}$ where $\dot{r}_{\beta_{\alpha}}$ denotes the name of
the $\beta_{\alpha}$-generic real. It follows by genericity that $\mathcal{A}$
will contain an $\R$-embeddable subfamily of size $\omega_{2}.$
\end{proof}


The rest of this section is devoted to the study of the controlled version
of the $\R$-embeddability in the Sacks model. In Theorem \ref{controlled-main}
we  obtain the maximal possible
$\omega_1$-controlled  embedding property since no family of
size $\mathfrak c$ can have $\mathfrak c$-controlled $\R$-embedding property by
Theorem \ref{no-c-controlled}.

\begin{definition} $e: 2^\omega\rightarrow 2^\omega$ is the function
satisfying $e(x)(n)=x(2n)$ for every $n\in \omega$.

\end{definition}

\begin{lemma}\label{nice-e} Let $u\subseteq 2^{<\omega}$ be in $\SSS$
 and $H:[u]\rightarrow 2^\N$
be a homeomorphism.
Let $\alpha<\omega_2$. Whenever $p\in  \SSS_{\omega_2}$
is such that $p\upharpoonright\alpha\forces p(\alpha)=\check u$ 
and $p\upharpoonright\alpha \forces \dot x\in 2^\omega$ for an $\SSS_\alpha$-name $\dot x$, 
then there is an $\SSS_\alpha$-name $\dot q$ such that 
$(p\upharpoonright\alpha)^\frown \dot q\in \SSS_{\alpha+1}$, 
$(p\upharpoonright\alpha)^\frown \dot q\leq p\upharpoonright(\alpha+1)$ and
$$(p\upharpoonright\alpha)^\frown \dot q\forces \check{ e\circ H}(\dot s_\alpha)=\dot x,$$
In particular $(p\upharpoonright\alpha)^\frown \dot q\forces \check e(\dot s_\alpha)=\dot x$, if
$p(\alpha)=1_{\SSS}$.
\end{lemma}
\begin{proof}
Define $\dot q$ to be an $\SSS_\alpha$-name for the set
$$\{y\upharpoonright n\mid y\in [u],  \ \forall k\in \omega\ H(y)(2k)=x(k), n\in \omega\}.$$
 This is an $\SSS_\alpha$-name for a perfect subtree
of $u$ and so $(p\upharpoonright\alpha)^\frown \dot q\in \SSS_{\alpha+1}$,
 $(p\upharpoonright\alpha)^\frown \dot q\leq p\upharpoonright(\alpha+1)$. We also have
$(p\upharpoonright\alpha)^\frown \dot q\forces \dot s_\alpha\in \dot q$ and $e(H(z))=x$ for every
$z\in [q]$, so the lemma follows.
\end{proof}

\begin{lemma}\label{p-nice-e} Let $\beta<\delta<\omega_2$ and suppose that
$p\in \SSS_{\delta+1}\subseteq \SSS_{\omega_2}$ and $\dot F$ is an $\SSS_{\delta}$-name 
for a continuous function from $2^\omega$ onto $2^\omega$ such that $F^{-1}[\{x\}]\cap [p(\delta)]$ 
is perfect for every $x\in 2^\omega$ in any forcing extension.
There is an $\SSS_{\delta}$-name $\dot r$
such that $p\upharpoonright\delta^\frown \dot r\leq p$ and
$$p\upharpoonright\delta^\frown \dot r\forces \dot F(\dot s_\delta)=\dot s_\beta.$$
\end{lemma}
\begin{proof} Let $\dot q$ be  an $\SSS_{\delta}$-name for the set
$$\bigcap_{u\in \GG_\delta} F^{-1}[[u(\beta)]]\cap [p(\delta)]=
 F^{-1}[\bigcap_{u\in \GG_\delta}[u(\beta)]]\cap [p(\delta)]= 
 F^{-1}[\{s_\beta\}]\cap p(\delta).$$
It is a name for a perfect set, as preimages of singletons under $F$ are perfect in  $p(\delta)$
in any forcing extension. Let $\dot r$ be such a name that $[r]=q$.
So $p\upharpoonright\delta^\frown \dot r\in \SSS_{\delta+1}$. Also $q\subseteq [p(\delta)]$, so
$p\upharpoonright\delta^\frown \dot r\leq p$. If $z\in q$, then 
$F(z)=\dot s_\beta$. But $p\upharpoonright\delta^\frown \dot r\forces
\dot s_{\delta}\in [\dot r]=\dot q$, so the lemma follows.

\end{proof}

\begin{definition}\label{coding}$c_1: \SSS\rightarrow {2^\omega}$ is the following coding of
perfect subtrees of $2^{<\omega}$ by the reals. Let $\tau: \N\rightarrow 2^{<\omega}$ be any
fixed bijection. Then given  $p\in\SSS$ we define
$c_1(p)(n)=1$ if and only if  $\tau(n)\in p$.
$c_2$ will denote the decoding function i.e., 
$$c_2(x)=\{ \tau(n)\mid  x(n)=1, \ n\in \omega\}.$$
\end{definition}

\begin{definition}\label{def-p-surjection}
Let $\{U_n\mid n\in \N\}$ be a fixed bijective enumeration of all 
 clopen subsets of $2^\omega$. 
Suppose that $p\in\SSS$.
Define $F_p: 2^\omega\rightarrow 2^\omega$ as follows: First by recursion define
a strictly increasing sequence $(n_i)_{i\in \N}$ such that $n_0$ is minimal satisfying
$U_{n_0}\cap [p]\not=\emptyset\not= [p]\setminus U_{n_0}$
and both $U_{n_0}$ and $2^\omega\setminus U_{n_0}$ are intervals
 in the lexicographical order on $2^\omega$. Given $n_0, ..., n_k$
for $k\in \N$ let $n_{k+1}$ be minimal such that $n_{k+1}>n_k$ and
the following  conditions hold for every $\sigma\in 2^{k+2}$:
\begin{enumerate} 
\item $\bigcap_{0\leq i\leq k+1} U_{n_i}^{\sigma(i)}$ is an
 interval in the lexicographical order on $2^\omega$,
\item $\bigcap_{0\leq i\leq k+1} U_{n_i}^{\sigma(i)}\cap [p]\not=\emptyset,$
\item $diam(\bigcap_{0\leq i\leq k+1} U_{n_i}^{\sigma(i)}\cap [p])\leq (2/3)^{k+1}$,
\end{enumerate}
where $V^1=V$ and $V^{0}=2^\omega\setminus V$.
Finally for $x\in 2^\omega$ and $i\in \omega$ we define
$$F_p(x)(i)=\chi_{U_{n_{2i}}}(x).$$
\end{definition}

\begin{lemma}\label{perfect-absolute}
Let $p\in\SSS$.  $F^{-1}_p[\{x\}]\cap [p]$ is perfect
for any $x\in2^\omega$ in any forcing extension.
\end{lemma}
\begin{proof} The conditions (1) - (3) of Definition \ref{def-p-surjection} guarantee the
property in the statement of the lemma, but they are preserved by any forcing.
\end{proof}

\begin{lemma}\label{continuity}
The function $f:2^\omega\times 2^\omega\rightarrow 2^\omega$  defined as
$$f(x, y)=F_{c_2(x)}(y)$$
is continuous.
\end{lemma}
\begin{proof} Let $\varepsilon>0$.
Let $\{U_n\mid n\in \N\}$ and $\{U_{n_i}\mid i\in \N\}$ be as in Definition \ref{def-p-surjection}.
Let $i_0\in 2\N$ be such that $\Sigma_{i=i_0}^\infty1/2^i<\varepsilon/2$. 
Given $p$ there is $m\in\omega$ such that if $p, p'$ are perfect subsets of $2^\omega$ such that
$c_1(p)\upharpoonright m=c_1(p')\upharpoonright m$, then the constructions of
 $\{U_{n_i}\mid i<i_0\}$ for $p$ and $p'$ agree. It follows that if $x_n$ is sufficiently
 close to $x$, then $|F_{c_2(x_n)}(z)-F_{c_2(x)}(z)|<\varepsilon/2$ (i.e.,
$F_{c_2(x_n)}$ converges uniformly to $F_{c_2(x)}$).
So 
$$|F_{c_2(x)}(y)-F_{c_2(x_n)}(y_n)|=
|F_{c_2(x)}(y)-F_{c_2(x)}(y_n)+F_{c_2(x)}(y_n)-F_{c_2(x_n)}(y_n)|\leq$$
$$\leq |F_{c_2(x)}(y)-F_{c_2(x)}(y_n)|+|F_{c_2(x)}(y_n)-F_{c_2(x_n)}(y_n)|<\varepsilon$$
 if $|y-y_n|$ and $|x-x_n|$ are sufficiently small by the 
 continuity of $F_{c_2(x)}$ and the above-mentioned uniform convergence.
\end{proof}

\begin{theorem}\label{reversing} The following statement is true in the Sacks model:
Suppose that $\{x_\xi\mid \xi<\omega_2\}\subseteq 2^\omega$ is a set of distinct reals
and $\{y_\xi\mid \xi<\omega_2\}\subseteq 2^\omega$.  Then there is a continuous 
$g: 2^\omega\rightarrow 2^\omega$
and $X\subseteq \omega_2$ of cardinality $\omega_1$ such that $g(x_\xi)=y_\xi$
for all $\xi\in X$. In fact, there is a ground model  continuous
 $\phi:2^\omega\times 2^\omega \rightarrow 2^\omega$
such that $\phi(x_\xi, s_\delta)=y_\xi$ for all $\xi\in X$ and some $\delta<\omega_2$.
\end{theorem}
\begin{proof}
As {\sf CH} holds in intermediate models we may assume that there are strictly increasing
$\{\beta_\theta\mid \theta<\omega_2\}$, conditions
 $p_\theta\in \SSS_{\beta_\theta}\subseteq \SSS_{\omega_2}$
and $\SSS_{\beta_\theta}$-names $\dot x_\theta, \dot y_\theta$ for $x_\theta$ 
and $y_\theta$ respectively where $\theta<\omega_2$ such that
$p_\theta\forces \dot x_\theta\not\in V^{\SSS_{\theta+1}}$. Using the {\sf CH} in the ground model
we can apply the stationary $\Delta$-system lemma\footnote{By the stationary $\Delta$-system lemma
we will mean the following lemma: given a family $\{X_\theta\mid \theta<\omega_2\}$
of countable subsets of $\omega_2$ there is a stationary set 
$A\subseteq\{\alpha\in\omega_2\mid cf(\alpha)=\omega_1\}$
such that $\{X_\theta\mid \theta\in A\}$ forms a $\Delta$-system. One can prove it as follows: 
Take regressive $f: \{\theta<\omega_2\mid cf(\theta)=\omega_1\}\rightarrow \omega_2$
given by $f(\theta)=\sup(X_\theta\cap \theta)$. Use the pressing down lemma obtaing 
a stationary $A'\subseteq A$ where $f$ is constantly equal to $\theta_0$. By {\sf CH}  and the $\omega_1$-additivity of the
nonstationary ideal on $\omega_2$ there is a stationary  $A''\subseteq A'$
such that $X_\theta\cap \theta_0$ is constant for $\theta\in A''$. 
Consider $g:\omega_2\rightarrow \omega_2$
given by $g(\theta)=\sup\{\sup(X_\eta)\mid \eta\leq\theta\}$. Let $A\subseteq A''$ 
 be the intersection of $A''$ with the club consisting of the ordinals bigger than $\theta_0$
 and  closed under 
 $g$. $A$ is the required set.} for countable sets and obtain a stationary $A\subseteq\{\alpha\in\omega_2: cf(\alpha)=\omega_1\}$
 such that $\{supp(p_\theta)\mid \theta\in A\}$ forms a $\Delta$-system
with root $\Delta\subseteq\omega_2$ and all the conditions agree on $\Delta$. 


We can use the result of \cite{hart} to find continuous
$h_\theta: 2^\omega\rightarrow 2^\omega$ and $q_\theta\leq p_\theta$ such that 
$q_\theta\forces 
\check h_\theta(\dot x_\theta)=\dot s_{\theta},$
for all $\theta\in A$.  Use the pressing down lemma finding 
a stationary $A'\subseteq A$ such that there is $\alpha<\omega_2$
with $supp(q_\theta)\cap\theta\subseteq \alpha$ for all $\theta\in A'$.
We will work for the rest of the proof in  $V^{\SSS^\alpha}$ which will
be treated as the ground model. 
By passing to a subset of $A'$ of cardinality $\omega_2$ 
and renaming the $q_\theta$'s we may assume that
$$p_\theta\forces 
\check h(\dot x_\theta)=\dot s_{\theta},\leqno (1)$$
for a fixed continuous $h: 2^\omega\rightarrow 2^\omega$ and all $\theta\in A'$
and $p(\theta)$ is a fixed perfect tree $u\subseteq 2^{<\omega}$ and the supports of
$p_\theta$s for $\theta\in A'$ are pairwise disjoint and $\min(supp(p_\theta))=\theta$
for all $\theta\in A'$. Also
fix a homeomorphism $H: [u]\rightarrow 2^\omega$. Construct a strictly increasing
$\{\theta_\xi\mid \xi<\omega_1\}$ such that ${\theta_\xi}<\beta_{\theta_\xi}<
{\theta_{\xi'}}$ for all $\xi<\xi'<\omega_1$. Relabel the involved objects as
$p_\xi:=p_{\theta_\xi}$, $\alpha_\xi:={\theta_\xi}$, $\beta_\xi:=\beta_{\theta_\xi}$, 
$\dot x_\xi:=\dot x_{\theta_\xi}$,
$\dot y_\xi:=\dot y_{\theta_\xi}$. Let $\delta<\omega_2$ be 
$\sup\{\alpha_\xi:\xi<\omega_1\}=\sup\{\beta_\xi\dot\xi<\omega_1\}$.

We will work with the iteration $\SSS_{\delta+1}$. In the model $V^{\SSS_{\delta+1}}$
 $g$ is defined by 
 $$g(x)=e\circ f(e\circ H(h(x)), s_{\delta}),$$
  where $f$ is as in  Lemma \ref{continuity}.
 By (1) it is enough to prove that given $p\in \SSS_{\delta+1}$ and $\xi<\omega_1$
  there is $p'\leq p$,  $p'\in \SSS_{\delta+1}$ and $\xi<\xi'<\omega_1$ such that
  $$p'\forces \dot f( e\circ H(\dot s_{\alpha_{\xi'}}), \dot s_{\delta})=\dot s_{\beta_{\xi'}},
  \ \ e(\dot s_{\beta_{\xi'}})=\dot {y_\xi}.\leqno (2)$$
  
  Let $\xi<\xi'<\omega_1$ be such that the support of $p\upharpoonright\delta$ is included in
  $\alpha_{\xi'}$, so we can assume that $p\upharpoonright\delta\in \SSS_{\alpha_{\xi'}}$ and
  so $p(\delta)$ is an $\SSS_{\alpha_{\xi'}}$-name. As 
  $supp(p_{\xi'})\subseteq [\alpha_{\xi'},\beta_{\xi'})$, the conditions $p$ and $p_{\xi'}$
  are compatible. Let $p''\in \SSS_{\delta+1}$ be obtained from $p$ by replacing $1$ by
  $p_{\xi'}(\alpha)$ on any $\alpha\in [\alpha_{\xi'},\beta_{\xi'})$ so that 
  $p''\leq p, p_{\xi'}$ and $p''(\alpha_{\xi'})=u$. Now to obtain the desired $p'\leq p''$
  we will modify $p''$ on $\alpha_{\xi'}, \beta_{\xi'}$ and $\delta$ using Lemmas
  \ref{nice-e}  and \ref{p-nice-e}.
  
  By Lemma \ref{nice-e}   there is an $\SSS_{\alpha_{\xi'}}$-name $\dot q$ such that 
$(p''\upharpoonright\alpha_{\xi'})^\frown \dot q\in \SSS_{\alpha_{\xi'}+1}$,  
$(p''\upharpoonright\alpha_{\xi'})^\frown \dot q
\leq p''\upharpoonright(\alpha_{\xi'}+1)$ and
$$(p''\upharpoonright\alpha_{\xi'})^\frown
 \dot q\forces \check e\circ \check H(\dot s_{\alpha_{\xi'}})={\check c_1}( p(\delta)).\leqno (3)$$
 
Since $p''(\beta_{\xi'})=1$
and $y_{\xi'}$ is an $\SSS_{\beta_{\xi'}}$-name by the last part of Lemma \ref{nice-e}
   there is an $\SSS_{\beta_{\xi'}}$-name $\dot o$ such that 
$(p''\upharpoonright\beta_{\xi'})^\frown \dot o\in \SSS_{\beta_{\xi'}+1}$, 
 $(p''\upharpoonright\beta_{\xi'})^\frown \dot o
\leq p''\upharpoonright(\beta_{\xi'}+1)$ and
$$(p''\upharpoonright\beta_{\xi'})^\frown
 \dot o\forces \check e(\dot s_{\beta_{\xi'}})=\dot y_{\xi'}.\leqno (4)$$

In $V^{\SSS_{\beta_{\xi'}}}$ consider the continuous
function $F_{p(\delta)}$ as defined in Definition 
\ref{def-p-surjection}. Apply Lemma \ref{p-nice-e} whose hypothesis is satisfied
by Lemma \ref{perfect-absolute} finding
 an $\SSS_{\delta}$-name $\dot r$
such that $p''\upharpoonright\delta^\frown \dot r\leq p''$ and
$$p''\upharpoonright\delta^\frown \dot r\forces \dot F_{p(\delta)}(\dot s_\delta)
=\dot s_{\beta_{\xi'}}.\leqno (5)$$

Define $p'\leq p$ in $\SSS_{\delta+1}$  by replacing in $p''$
 \begin{itemize}
 \item $u$ by $\dot q$  on the $\alpha_{\xi'}$-th coordinate,
 \item $1$ by $\dot o$   on the $\beta_{\xi'}$-th coordinate,
 \item $p(\delta)$ by $\dot r$ on the $\delta$-th coordinate.
 \end{itemize}
 It follows that $p'\in \SSS_{\delta}$, $p'\leq p''\leq  p$ and
 $p'\upharpoonright(\alpha_{\xi'}+1)\leq (p''\upharpoonright\alpha_{\xi'})^\frown \dot q$, 
 $p'\upharpoonright(\beta_{\xi'}+1)\leq (p''\upharpoonright\alpha_{\xi'})^\frown \dot o$
 and $p'\upharpoonright(\delta+1)\leq (p''\upharpoonright\delta)^\frown \dot r$. 
 
 Note that (5) and (3)  gives that
$$p'\forces \dot F_{c_2(\check e\circ \check H(s_{\alpha_{\xi'}}))}(\dot s_\delta)=
\dot F_{p(\delta)}(\dot s_\delta)
=\dot s_{\beta_{\xi'}}.$$
which together with (4) gives the required (2).
 \end{proof}

\begin{remark} It is proved in \cite{hart} that under the hypothesis of Proposition \ref{reversing}
 there is  a continuous $g: 2^\N\rightarrow 2^\N$
and either there is $X\subseteq\omega_2$ of cardinality $\omega_2$ such that $g(x_\xi)=y_\xi$ or
$g(y_\xi)=x_\xi$. 
Note that if $x_\xi=s_\xi$ and $y_\xi=s_{\xi+1}$, where $s_\xi$ denotes the $\xi$-th Sacks real
for $\xi<\omega_2$, then there is 
 there is no continuous $g: 2^\N\rightarrow 2^\N$
 such that $g(x_\xi)=y_\xi$ for $\omega_2$-many $\xi<\omega_2$. This follows from the fact that any continuous function is coded
 in some intermediate model. 

\end{remark}

\begin{theorem}\label{controlled-main} In the Sacks model every almost disjoint family of cardinality $\omega_2$
has the $\omega_1$-controlled embedding property. 
\end{theorem}
\begin{proof} Work in the Sacks model. Let $\A$ be any almost disjoint family of cardinality
$2^\omega=\omega_2$ and $\phi: \A\rightarrow 2^\omega$ any function. By Theorem  \ref{sacks-main}
and Lemma \ref{equivalences} and Remark \ref{remark-embedding} there is a subfamily $\A'\subseteq \A$ of cardinality 
$\omega_2$  and a function $f: \A'\rightarrow 2^\omega$,
such that the limits  $x_A=\lim_{n\in A}f(n)$ exist for each $A\in \A'$ and are different for
distinct $A\in \A'$.
  By Theorem \ref{reversing} there is a subfamily
$\mathcal B\subseteq \A'$ of cardinality $\omega_1$ and a continuous $g:2^\omega\rightarrow2^\omega$ 
such that $g(x_A)=\phi(A)$ for all $A\in \mathcal B$. By the continuity
of $g$ we have $\phi(A)=g(x_A)=\lim_{n\in A}g(f(n))$ 
for all $A\in \mathcal B$. So $f':\N\rightarrow 2^\omega$ given by $g\circ f$ witnesses
the $\omega_1$-controlled embedding property for $\A$ and $\phi$.

\end{proof}

\section{An application: Abelian subalgebras of Akemann-Doner C*-algebras}

The application of our combinatorial results from the previous sections 
presented here is related to noncommutative C*-algebras defined by C. Akemann and J. Doner in \cite{akemann-doner}
with the help of almost disjoint families. Let us recall these constructions.
We consider the C*-algebra  $M_2$ of all complex $2\times 2$ matrices with the usual operations
like in linear algebra and with the linear operator norm. In this section $\C$ will
stand for the field of complex numbers.
By $\ell_\infty(M_2)$ we denote the C*-algebra of
all norm bounded sequences from $M_2$ with the supremum norm and the coordinatewise operations.
By $c_0(M_2)$ we denote the C*-subalgebra of $\ell_\infty(M_2)$
consisting of sequences of matrices whose norms converge to zero.

For $\theta\in [0,2\pi)$ define a $2\times 2$ complex matrix of a rank one projection by
\[p_\theta=\begin{bmatrix}\sin^2\theta&\sin\theta\cos\theta\\ \sin\theta\cos\theta&\cos^2\theta\end{bmatrix}.\]
Given $A\subseteq \N$ and $\theta\in [0, 2\pi)$ define $P_{A, \theta}\in \ell_\infty(M_2)$ by
$$ P_{A, \theta}(n)=\begin{cases}
               0               & n\not\in A\\
               p_\theta               & n \in A
              
           \end{cases}$$
 Given an almost disjoint family $\A\subseteq \wp(\N)$ and a function
 $\phi: \A\rightarrow [0, 2\pi)$ the Akemann-Donner   algebra $AD(\A, \phi)$  is the subalgebra
 of $\ell_\infty(M_2)$ generated by $c_0(M_2)$ and $\{P_{A, \phi(A)}\mid A\in \A\}$.  
 As the distances between $P_{A, \theta}$ and $  P_{A', \theta'}$ are at least one for infinite and
 distinct $A, A'\subseteq \N$ and any $\theta, \theta'\in [0,2\pi)$, such algebras
 are nonseparable if $\A$ is uncountable.
Clearly if  $\A$ is uncountable and $\phi:\A\rightarrow [0, 2\pi)$ is constantly equal to $\theta$, then
$AD(\A, \phi)$ contains the nonseparable commutative C*-algebra
isomorphic to $C_0(\Psi(\A))$ of all complex valued continuous functions on $\Psi(\A)$ vanishing at infinity because $P_{\theta}^2=P_{\theta}=P_{\theta}^*$ since
it is a projection. However, as Akemann and Doner proved under {\sf CH}, one can choose
$\A$ so that for every  injective $\phi:\A\rightarrow (0, \pi/6)$  the algebra
$AD(\A, \phi)$ has no nonseparable commutative subalgebra. In \cite{tristan-piotr} the hypothesis
of {\sf CH} was  removed by showing that a {\sf ZFC} Luzin family $\A$ is sufficient for this result
of Akemann and Doner. We have the following two lemmas implicitly from
\cite{akemann-doner, tristan-piotr}:

\begin{lemma}\label{exists-commutative}
Suppose that $\A$ is an almost disjoint family and
 $\phi: \A\rightarrow [0, 2\pi)$.  If there is $\B\subseteq \A$ of cardinality $\kappa$ and $f:\N\rightarrow [0, 2\pi)$
 such that $\lim_{n\in B}f(n)=\phi(n)$ for every $B\in \B$, then $AD(\A, \phi)$ 
 contains a commutative C*-subalgebra of density $\kappa$. 
\end{lemma}
\begin{proof}
First define $P_f\in \ell_\infty(M_2)$ by
$P_{f}(n)=p_{f(n)}$.
 For $B\in \mathcal B$ define $R_B\in \ell_\infty(M_2)$ by $R_B(n)=P_{f}\chi_B(n)$,
 where $\chi_B$ is the characteristic function of $B$.
The hypothesis about $f$ implies that $R_B-P_{B, \phi(B)}\in c_0(M_2)$ and so $R_B$ is in $AD(\A, \phi)$.
 The algebra
generated by $\{R_B\mid B\in \mathcal B\}$ is commutative  isomorphic to
$C_0(\Psi(\B))$ and of density $\kappa$ as required.
\end{proof}

\begin{lemma}\label{nocommutative} Let $c\in \R$ be such that
$\|P_0-P_\theta\|<1/4$ for $\theta\in [0,c]$.
Suppose that $\A$ is an almost disjoint family and that
 $\phi: \A\rightarrow [0, c]$ is such that for no $\B\subseteq \A$ of cardinality $\kappa$ there is 
 $f:\N\rightarrow [0, c]$ 
 such that $\lim_{n\in A}f(n)=\phi(A)$ for every $A\in \B$. Then
  $AD(\A, \phi)$ 
 does not contain any commutative C*-subalgebra of density $\kappa$. 
\end{lemma}
\begin{proof} This is a slight modification of an argument from \cite{akemann-doner} 
and modified in  \cite{tristan-piotr}.
Let $\rho:\{P_\theta\mid \theta\in [0, 1/4]\}\rightarrow [0, 1/4]$ be defined
by $\rho(P_\theta)=\theta$. $\rho$ is a continuous map from a closed
subset of the unit ball $\BB_1$ of $M_2$ into $[0,1/4]$. Use the Tietze extension theorem to
find a continuous 
 $\eta: \BB_1\rightarrow [0,1/4]$ which extends $\rho$.

Suppose that $\mathcal C$ is a commutative subalgebra of $AD(\A, \phi)$
whose density is $\kappa$. As in \cite{akemann-doner} and \cite{tristan-piotr},
 in a slightly different language, it follows  from 
 simultaneous diagonalization of commuting matrices that
there are rank one projections $q(n)\in M_2$ such that $a(n)q(n)=q(n)a(n)$ for each
$n\in \N$ and each $a\in \mathcal C$ and we may assume  that $\|q(n)-P_0\|^2\leq1/2$ by
(2.1.) of \cite{tristan-piotr}. 
It is easy to note that for each element $a$ of $AD(\A, \phi)$ 
the limit $lim_{n\in A}a(n)$ exists and is equal to a multiple of $p_{\phi(A)}$.
The density  of $\mathcal C$ being $\kappa$ means that 
there is $\B\subseteq \A$ of cardinality $\kappa$ such that
for each $B\in \B$ there is $a_B\in \mathcal C$ such that the limit $lim_{n\in B}a_B(n)$ exists and is equal  $z_B p_{\phi(B)}$ for a nonzero complex number $z_B$.
By the compactness of the unit ball in $M_2$
for each infinite $B'\subseteq B$ there is an infinite $B''\subseteq B'$
such that $\lim_{n\in B''} q(n)=q'$ exists, and so it needs to be a rank one projection
which commutes with $lim_{n\in A}a_B(n)$ which is $z_B p_{\phi(B)}$, so $p_{\phi(B)}$
and $q'$ commute but $\|q'-p_{\phi(B)}\|\leq 1/\sqrt{2}+1/4<1$ and so $q'=p_{\phi(B)}$
(see e.g. Lemma 3 of \cite{tristan-piotr}).
This means that actually $\lim_{n\in B} q(n)$ exists and is equal to $p_{\phi(B)}$
for each $B\in \B$. By the continuity of $\eta$ we have
$\lim_{n\in B} \eta(q(n))=\eta(p_{\phi(B)})=\phi(B)$.
 Define  $f:\N\rightarrow [0,1/4]$ by $f(n)=\eta(q(n))$. So $\lim_{n\in B} f(n)=\phi(B)$ for
every $B\in \mathcal B$  contradicting the hypothesis on $\A$.
\end{proof}

As corollaries we obtain:

\begin{theorem}\label{no-c-com-zfc}
In ZFC, for every almost disjoint family $\A$ of cardinality $\mathfrak c$ there is
$\phi:\A\rightarrow [0,2\pi)$ such that the
 Akemann-Doner C*-algebra $AD(\A, \phi)$  of density  $\mathfrak c$  has no commutative subalgebras
of density  $\mathfrak c$.
\end{theorem}
\begin{proof} Fix an almost disjoint family $\A$ of cardinality $\mathfrak c$.
By Theorem \ref{no-c-controlled} there is $\phi: \A\rightarrow \R$ such that
for no $\B\subseteq \A$ of cardinality $\mathfrak c$ there is $f:\N\rightarrow \R$
such that $\lim_{n\in B}f(n)=\phi(B)$ for all $B\in \B$.
By applying a continuous injective mapping we may assume that  $\R$ is replaced by $[0,c]$,
where $c\in \R$ is like in Lemma \ref{nocommutative}. Now Lemma \ref{nocommutative}
implies that $AD(\A, \phi)$ has no commutative subalgebras
of density  $\mathfrak c$.

\end{proof}

\begin{theorem}\label{nonsep-com-sacks} It is consistent that every Akemann-Doner algebra
of density $\mathfrak c$ contains a nonseparable commutative subalgebra.
\end{theorem}
\begin{proof} We claim that the above statement holds in the Sacks model. By Theorem 
\ref{controlled-main} given any almost disjoint family $\A$ of cardinality $\mathfrak c$
and a functions $\phi:\A\rightarrow\R$ there is an uncountable $\B\subseteq \A$ 
such that $\lim_{n\in B} f(n)=\phi(B)$ for all $B\in \B$. It follows form
Lemma \ref{exists-commutative} that $AD(\A, \phi)$ contains a nonseparable commutative subalgebra.
\end{proof}

\begin{theorem}\label{no-nonsep-com-cohen}  Let $c\in \R$ be such that
$\|P_0-P_\theta\|<1/4$ for $\theta\in [0,c]$. It is consistent with the negation of 
{\sf CH} that for every almost disjoint family $\A$ of cardinality
  $\mathfrak c$ there is $\phi:\A\rightarrow [0,c]$
such that the Akemann-Doner algebra $AD(\A, \phi)$  of density  $\mathfrak c$
has no nonseparable commutative subalgebra.
\end{theorem}
\begin{proof} Work in the Cohen model.
Fix an almost disjoint family $\A$ of cardinality $\mathfrak c$.
By Theorem \ref{cohen-no-controlled} there is $\phi: \A\rightarrow \R$ such that
for no uncountable $\B\subseteq \A$  there is $f:\N\rightarrow \R$
such that $\lim_{n\in B}f(n)=\phi(B)$ for all $B\in \B$.
By applying a continuois mapping we may assume that the interval $\R$ is replaced by $[0,c]$,
where $c\in \R$ is like in Lemma \ref{nocommutative}. Now Lemma \ref{nocommutative}
implies that $AD(\A, \phi)$ has  no  commutative nonseparable subalgebras.
\end{proof}

\begin{theorem}\label{superno-nonsep-com-cohen}  Let $c\in \R$ be such that
$\|P_0-P_\theta\|<1/4$ for $\theta\in [0,c]$. It is consistent with the negation of 
{\sf CH} that there is an almost disjoint family $\A$ of cardinality
  $\mathfrak c$ such that for every $\phi:\A\rightarrow [0, c]$
 the Akemann-Doner algebra $AD(\A, \phi)$  of density  $\mathfrak c$
has no nonseparable commutative subalgebra.
\end{theorem}
\begin{proof} Work in the Cohen model.
Let $\A$ be an almost disjoint family of cardinality $\mathfrak c$ from Theorem \ref{cohen-main}.
By Theorem \ref{cohen-main} for no $\phi: \A\rightarrow \R$ there is
an uncountable $\B\subseteq \A$  and $f:\N\rightarrow \R$
such that $\lim_{n\in B}f(n)=\phi(B)$ for all $B\in \B$.
 Now Lemma \ref{nocommutative}
implies that $AD(\A, \phi)$ has  no  commutative nonseparable subalgebras.
\end{proof}

These results  complete earlier result of \cite{tristan-piotr} 
that there is in ZFC an almost
disjoint family $\A$ (any inseparable family) such that for every
$\phi:\A\rightarrow[0,c)$ the Akemann-Doner algebra  of density $\omega_1$ has no nonseparable commutative subalgebra.

\bibliographystyle{amsplain}

\end{document}